\documentclass[a4paper,11pt]{article}

\usepackage[utf8]{inputenc}
\usepackage[T1]{fontenc}
\usepackage[english]{babel}

\usepackage[tbtags]{amsmath}
\usepackage{amssymb}
\usepackage{amsfonts}
\usepackage{amsthm}
\usepackage{calrsfs}
\usepackage{array}
\usepackage{bbm}
\usepackage{stmaryrd}
\usepackage{upgreek}

\usepackage{mathtools}
\usepackage[svgnames]{xcolor}
\usepackage{hyperref}

\usepackage{epsfig}
\usepackage{graphicx}
\usepackage{pstricks}

\newcommand\C{{\cal {C}}}

\newcommand\dd{{d(0)}}

\newcommand{\ds}{\displaystyle}
\newcommand\dv{\mathop{\rm div}}

\newcommand{\eb}{{\epsilon_0}}
\newcommand{\ee}{{\bar\epsilon}}

\renewcommand{\H}{{\cal H}}
\renewcommand\iint{\displaystyle\int_{|y|<1}}

\renewcommand{\k}{k}

\newcommand{\m}[1]{\mathbbm{#1}}

\newcommand{\pnu}{\partial_\nu}

\newcommand{\q}[1]{\mathcal{#1}}

\newcommand\RR{{\cal R}}

\renewcommand\SS{{\cal S}}

\newcommand\supp{\mathop{\rm supp}}
\newcommand \ttt {{\bar t}}
\newcommand \tautau t

\newcommand \xixi x

\newcommand{\vc}[2]{\begin{pmatrix} #1\\#2\end{pmatrix}}

\newcommand\wto{\rightharpoonup}

\theoremstyle{plain}
\newtheorem{thm}{Theorem}
\newtheorem*{thm*}{Theorem}
\newtheorem{coro}[thm]{Corollary}

\newtheorem{prop}{Proposition}[section]
\newtheorem{cl}[prop]{Claim}

\newtheorem{lem}[prop]{Lemma}

\theoremstyle{definition}

\theoremstyle{remark}
\newtheorem*{nb}{Remark}

\makeatletter
\def\blfootnote{\xdef\@thefnmark{}\@footnotetext}
\makeatother

\title {\bf On the stability of the notion of non-characteristic point and blow-up profile for semilinear wave equations
\footnote{Both authors are supported by the ERC Advanced Grant no. 291214, BLOWDISOL. H.Z. is partially supported by ANR project ANAÉ ref. ANR-13-BS01-0010-03.}
}
\author{Frank Merle\\
{\it \small Universit\'e de Cergy Pontoise and IHES}\\
Hatem Zaag\\
{\it \small Université Paris 13, Sorbonne Paris Cité,}\\
{\it \small LAGA, CNRS (UMR 7539), F-93430, Villetaneuse, France.}}

\allowdisplaybreaks

\begin{document} 

\maketitle

\begin{abstract}
We consider a blow-up solution for the semilinear wave equation in $N$ dimensions, with subconformal power nonlinearity. Introducing $\RR_0$ the set of non-characteristic points with the Lorentz transform of the space-independent solution as asymptotic profile, we show that $\RR_0$ is open and that the blow-up surface is of class $C^1$ on $\RR_0$. Then, we show the stability of $\RR_0$ with respect to initial data. 
\end{abstract}

\medskip


%


\section{Introduction}
We consider the subconformal focusing semilinear wave equations in $\m R^N$:
\begin{equation}\label{equ}
\left\{
\begin{array}{l}
\partial_t^2 u =\Delta u+|u|^{p-1}u,\\
(u(0), \partial_t u(0))\in H^1\times L^2,
\end{array}
\right.
\end{equation}
with 
\begin{equation}\label{condp}
1<p\mbox{ and }p< \frac{N+3}{N-1}\mbox{ if } N\ge 2.
\end{equation}
The Cauchy problem is locally wellposed. By energy arguments, Levine showed in \cite{Ltams74} the existence of blow-up solutions. 

\medskip

Equation \eqref{equ} can be considered as a lab model for blow-up in hyperbolic equations, because it captures features common to a whole range of blow-up problems arising in various nonlinear physical models, in particular in general relativity (see Donninger, Schlag and Soffer \cite{DSScmp12}), and also for self-focusing waves in nonlinear optics (see Bizo\'n, Chmaj and Szpak \cite{BCSjmp11}). 


\bigskip

The one-dimensional case in equation \eqref{equ} has been understood completely, in a series of papers by the authors \cite{MZjfa07, MZcmp08, MZxedp10, MZajm11, MZisol10} and in C\^ote and Zaag \cite{CZmulti11}.
See Caffarelli and Friedman in \cite{CFtams86} and \cite{CFarma85} for earlier results. See also Killip and Vi\c san \cite{KV11}. 

\bigskip

In higher dimensions, only the blow-up rate is known (see \cite{MZajm03}, \cite{MZma05} and \cite{MZimrn05}; see also the extension by Hamza and Zaag in \cite{HZnonl12} and \cite{HZjhde12} to the Klein-Gordon equation and other damped lower-order perturbations of equation \eqref{equ}). In fact, the program of dimension $1$ breaks down because there is no classification of selfsimilar solutions of equation \eqref{equ} in the energy space, except in the radial case outside the origin (see \cite{MZbsm11}). Considering the behavior of radial solutions at the origin, Donninger and Sch{\"o}rkhuber were able to prove the stability of the space-independent solution (i.e. the solution of the associated ODE $u"=u^p$) with respect to perturbations in initial data, in the Sobolev subcritical range \cite{DSdpde12} and also in the supercritical range in \cite{DStams14}. Some numerical results are available in a series of papers by Bizo\'n and co-authors (see \cite{Bjnmp01}, \cite{BCTnonl04}, \cite{BZnonl09}).\\
In this paper, we aim at extending to higher dimensions, the results of classification of the blow-up behavior proved in one space dimension in \cite{MZcmp08}.

\medskip

If $u$ is a blow-up solution of equation \eqref{equ}, we define (see for example Alinhac \cite{Apndeta95}) a 1-Lipschitz surface $\{(x,T(x))\}$ where $x\in{\m R}^N$ such that the domain of definition of $u$ is written as 
\begin{equation}\label{defdu}
D=\{(x,t)\;|\; t< T(x)\}.
\end{equation}
$\{(x,T(x))\}$ is called the blow-up surface of $u$. 
A point $x_0\in{\m R}^N$ is a non-characteristic point if there are 
\begin{equation}\label{nonchar}
\delta_0\in(0,1)\mbox{ and }t_0<T(x_0)\mbox{ such that }
u\;\;\mbox{is defined on }{\cal C}_{x_0, T(x_0), \delta_0}\cap \{t\ge t_0\}
\end{equation}
where 
\begin{equation}\label{defcone}
{\cal C}_{\bar x, \bar t, \bar \delta}=\{(x,t)\;|\; t< \bar t-\bar \delta|x-\bar x|\},
\end{equation}
as illustrated in figure \ref{fig1}.
 \begin{figure}
\centering
\includegraphics[width=0.6\textwidth]{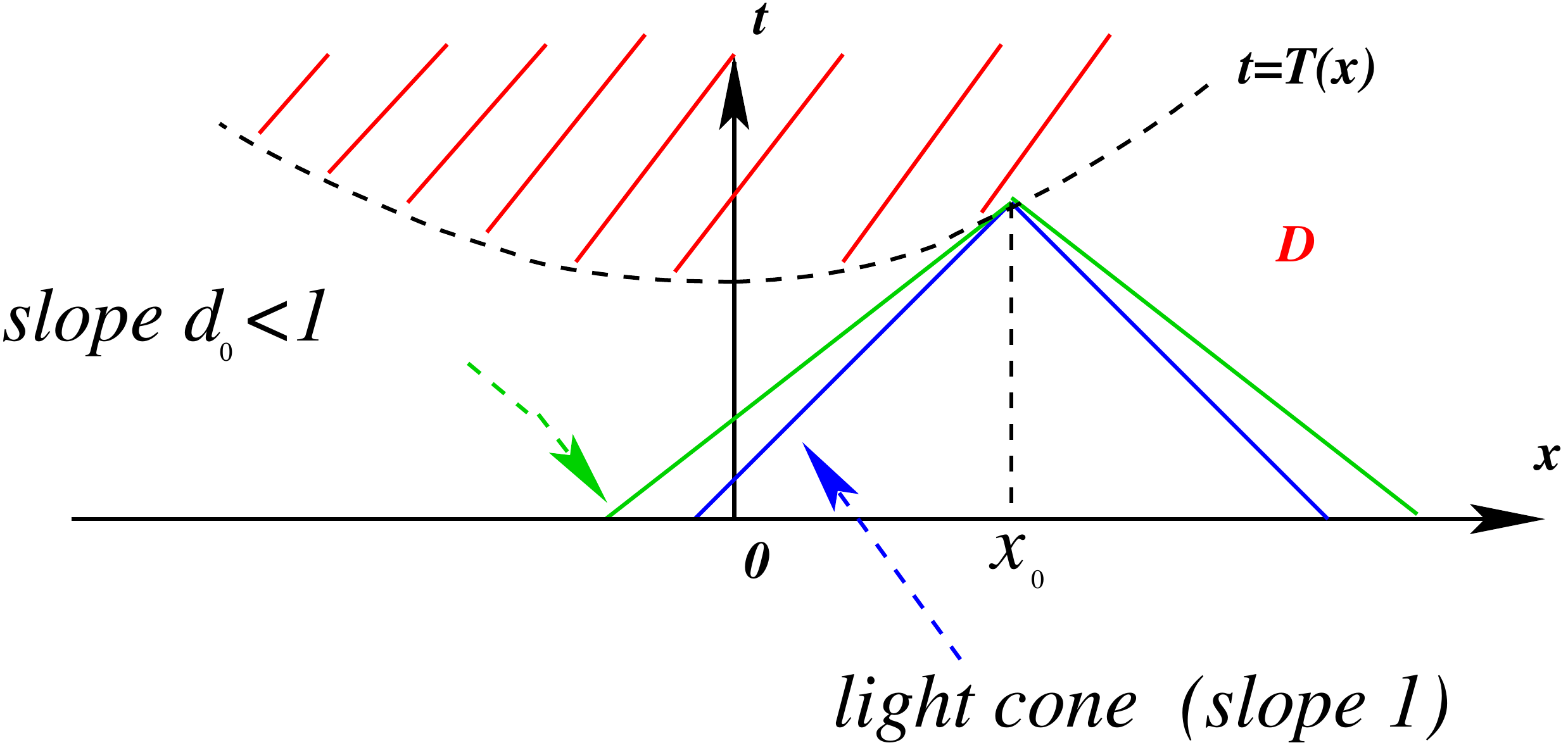}
\caption{$x_0$ is a non-characteristic point ($N=1$).}\label{fig1}
\end{figure}
If not, we say that $x_0$ is a characteristic point.
 We denote by $\RR\subset {\m R}^N$ the set of non-characteristic points and by $\SS$ the set of characteristic points. 
In \cite{MZajm11}, we showed the existence of solutions with $\SS \neq \emptyset$, in dimensions $N\ge 1$. In \cite{MZisol10}, we proved that $\SS$ is locally finite, and that the blow-up surface $\{(x,T(x))\}$ is tangent to the backward light-cone near any $x_0\in \SS$, as illustrated in figure \ref{fig2}. It is an open problem to tell whether this holds also in higher dimensions or not.
\begin{figure}
\centering
\includegraphics[width=0.6\textwidth]{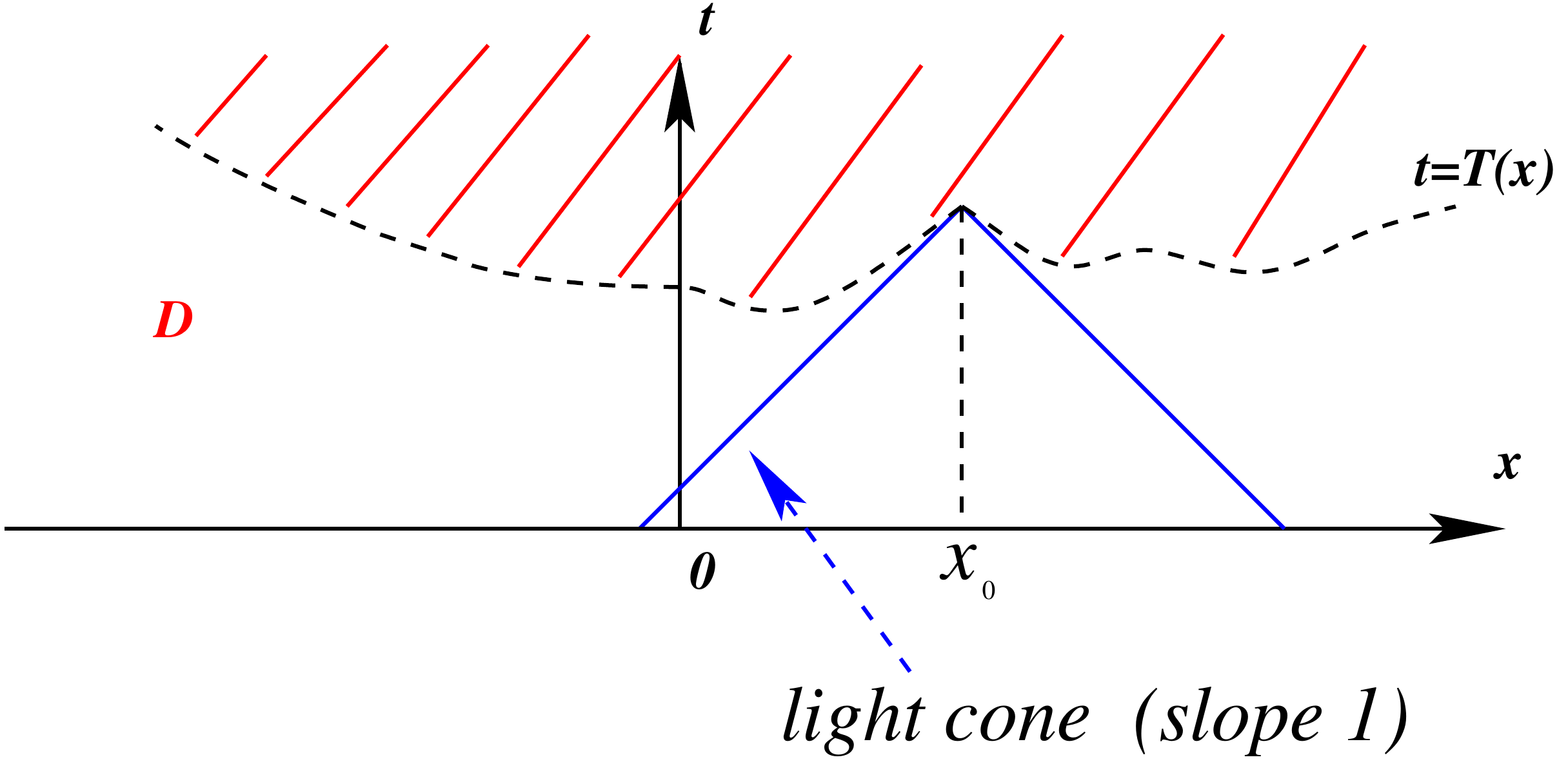}
\caption{$x_0$ is a characteristic point ($N=1$).}\label{fig2}
\end{figure}

\bigskip

Let us introduce the following similarity variables, for any $(x_0,T_0)$ such that $0< T_0\le T(x_0)$:
\begin{equation}\label{defw}
w_{x_0,T_0}(y,s)=(T_0-t)^{\frac 2{p-1}}u(x,t),\;\;y=\frac{x-x_0}{T_0-t},\;\;
s=-\log(T_0-t).
\end{equation}
If $T_0=T(x_0)$, we write $w_{x_0}$ for short. 
The function $w_{x_0,T_0}$ (we write $w$ for simplicity) satisfies the 
following equation for all $y\in B(0,1)$ and $s\ge -\log T_0$:
\begin{equation}\label{eqw}
\partial^2_sw-\q L w+\frac{2(p+1)}{(p-1)^2}w-|w|^{p-1}w=
-\frac{p+3}{p-1}\partial_sw-2y\cdot\nabla \partial_s w
\end{equation} 
where
\begin{equation}\label{defro}
\q L w =\frac 1\rho \dv \left(\rho \nabla w-\rho(y\cdot \nabla w)y\right),\;\; \rho(y)=(1-|y|^2)^\alpha\mbox{ and }\alpha=\frac 2{p-1}-\frac{N-1}2>0.
\end{equation}
Equation \eqref{eqw} is studied in the energy space
\begin{equation}\label{defnh}
\H = \left\{(q_1,q_2) 
\;\;|\;\;\|(q_1,q_2)\|_{\H}^2\equiv \iint \left(q_1^2+|\nabla q_1|^2-|y\cdot \nabla q_1|^2)+q_2^2\right)\rho dy<+\infty\right\}.
\end{equation}
We also introduce for all $|d|<1$ the following stationary solutions of \eqref{eqw} (or solitons) defined by 
\begin{equation}\label{defkd}
\kappa(d,y)=\kappa_0 \frac{(1-|d|^2)^{\frac 1{p-1}}}{(1+d\cdot y)^{\frac 2{p-1}}}\mbox{ where }\kappa_0 = \left(\frac{2(p+1)}{(p-1)^2}\right)^{\frac 1{p-1}} \mbox{ and }|y|<1.
\end{equation}

\bigskip

Let us first review our results in one-space dimension
in the non-characteristic case.

\bigskip

In \cite{CZmulti11, MZjfa07, MZcmp08, MZxedp10, MZajm11, MZisol10}, we give an exhaustive description of the geometry of the blow-up curve on the one hand, and the asymptotic behavior of solutions near the blow-up curve on the other hand
(see also \cite{MZbsm11} and \cite{HZbsm13}).
 Focusing on the non-characteristic case, we recall the following results:\\
{\it - The geometry of the blow-up surface}: $\RR$ is a non empty open set, and $x\mapsto T(x)$ is of class $C^1$ on $\RR$ (see Theorem 1 page 58 and the following remark in \cite{MZcmp08});\\
{\it - The blow-up behavior near $(x_0, T(x_0))$ when $x_0\in \RR$}: it holds that $w_{x_0}(s) \to \pm \kappa(T'(x_0))$ as $s\to \infty$ (see Corollary 4 page 49 in \cite{MZjfa07}).

\bigskip

A key step in this program was to show that the set of all stationary solutions of equation \eqref{eqw} in the energy space is
\begin{equation*}
\{0;\pm \kappa(d,y)\:|\;|d|<1\}.
\end{equation*}
Unfortunately, in higher dimensions, we have been unable to determine that set, though one can trivially see that $0$ and $\pm \kappa(d,\cdot)$ for all $|d|<1$ are still stationary solution of \eqref{eqw}.
Note that when $N=3$ and $p=3$ or $p\ge 7$ is an odd integer, Bizo\'n, Breitenlohner, Maison and Wasserman proved in \cite{BMWnonl07} and \cite{BBMWnonl10} the existence of a countable family of radially-symmetric stationary solutions.

\bigskip

In this paper, we consider $x_0\in\RR$ and assume that 
the blow-up profile at $x_0\in\RR$ is given by $\pm \kappa(d(x_0),\cdot)$ as in one space dimension, in the sense that
\begin{equation}\label{profileN}
\vc{w_{x_0}(s)}{\partial_s w_{x_0}(s)}\to e(x_0)\vc{\kappa(d(x_0))}{0}
\mbox{ in }\q H
\mbox{ as }s\to \infty,
\end{equation}
for some $e(x_0)=\pm 1$ and $|d(x_0)|<1$. 
Then, assuming \eqref{profileN}, we will prove that for $x\sim x_0$, $x\in \RR$ and $w_x(s) \to e(x_0) \kappa(d(x))$ as $s\to \infty$. 
More precisely, let us introduce the set 
\begin{equation}\label{defr0}
\RR_0 = \{x_0\in \RR\;|\;\exists\; |d(x_0)|<1,\exists\;e(x_0)=\pm 1
\mbox{ s.t. }\eqref{profileN}\mbox{ holds}\}.
\end{equation}
Let us state our first result:
\begin{thm}[Openness of $\RR_0$ and regularity of the blow-up surface]\label{threg}$ $\\
 The set $\RR_0$ is open and the blow-up surface $x\mapsto T(x)$ is of class $C^1$ on $\RR_0$.\\
 Moreover, for all $x_0\in \RR_0$, $\nabla T(x_0)=d(x_0)$ and $e(x_0)$ is constant on connected components of $\RR_0$, where $d(x_0)$ and $e(x_0)$ are defined in \eqref{defr0}.
\end{thm}
\begin{nb} With the same proof as in one space dimension (see the remark following Theorem 1 in \cite{MZcmp08}), we know that $\RR\not= \emptyset$, for any blow-up solution. We don't know if such a result holds for $\RR_0$; we are only aware of some examples where $\RR_0\not=\emptyset$: either the explicit solutions such that initial data given by $\kappa(d,x)$ on some large enough ball, or also any of their perturbations, thanks to the stability of $\RR_0$ with respect to initial data stated in our following theorem \ref{thstab}. We expect that for some initial data, $\RR\backslash \RR_0\neq \emptyset$, in other words, that other stationary solutions of equation \eqref{equ}, different from $\pm\kappa(d,y)$ exist, as already proved by Bizo\'n and co-authors in \cite{BMWnonl07} and \cite{BBMWnonl10} when $N=3$ and $p=3$ or $p\ge 7$ is an odd integer. Nevertheless, we suspect those stationary solutions to be unstable. In other words, $\pm \kappa(d)$ for $|d|<1$ should be the only generic blow-up profiles. 
\end{nb}
\begin{nb} 
All the results in our paper are valid in the subconformal range $p<\frac{N+3}{N-1}$.\\
- The restriction to  $p<\frac{N+3}{N-1}$ comes from the existence in the similarity variables' setting \eqref{defw} of a Lyapunov functional (defined below in \eqref{defenergy}), which is localized in the backward light-cone, and which allows to control the blow-up rate and the whole blow-up dynamics (see \cite{MZajm03, MZimrn05, MZjfa07, MZajm11}).\\
- In the conformal case $p=\frac{N+3}{N-1}$, that Lyapunov functional persists (see \cite{MZma05}), though its dissipation degenerates to the boundary of the backward light-cone. Our results remain valid, 
at the expense of some more technical difficulties.\\
- If $\frac{N+3}{N-1}<p<\frac{N+2}{N-2}$, we don't know if other blow-up rates not given by the ODE are possible (see Killip, Stovall and Vi\c san \cite{KSVsurc12}, and Hamza and Zaag \cite{HZdcds13}, where an upper bound on the blow-up rate is given).\\
- In the Sobolev-critical case $p=\frac{N+2}{N-2}$, we know from Krieger, Schlag and Tataru \cite{KSTdmj09} as well as Hillairet and Rapha\"el \cite{HRapde12} that blow-up solutions with blow-up rates different from the solution of the ODE exist (see also Krieger, Nakanishi and Schlag \cite{KNSajm13, KNSdcds13}).\\
- If $p> \frac{N+2}{N-2}$, no results are available, not even on the blow-up rate question, which remains open. Let us mention that faster blow-up rates exist for the semilinear heat equation, for some exponents in this supercritical range (see Herrero and Vel\'azquez \cite{HVcras94}, Matano and Merle \cite{MMjfa09} and \cite{MMjfa11}).
\end{nb}
With almost the same proof as for Theorem \ref{threg}, we are able to derive the stability
of $\RR_0$  
with respect to initial data:
\begin{thm}[Stability of $\RR_0$ with respect to initial data]\label{thstab} 
Consider $\hat u(x,t)$ a blow-up solution of equation \eqref{equ} with initial data $(\hat u_0, \hat u_1)\in H^1\times L^2$ and $\hat x_0\in \hat \RR_0$. 
Then, there exists $\hat \epsilon_0>0$ such that for any $(u_0,u_1)$ with $\|(u_0,u_1)-(\hat u_0, \hat u_1)\|_{H^1\times L^2}\le \hat \epsilon_0$, the solution $u(x,t)$ of equation \eqref{equ} with initial data $(u_0,u_1)$ blows up in finite time with 
\[
B(\hat x_0, \hat \epsilon_0)\subset \RR_0.
\]
\end{thm}
\begin{nb} From the finite speed of propagation, we may use instead the $H^1\times L^2$ on a sufficiently large ball. We would like to point out that we show the continuity of the blow-up time with respect to initial data, near points in $\RR_0$. See Lemma \ref{lemcont} below.
\end{nb}
As a consequence of our results, we obtain the stability of the local minimal blow-up time, with respect to initial data, provided that the local minimum is achieved at some point in $\RR_0$. More precisely, this is our result:
\begin{coro}[Stability of the local minimal blow-up time]\label{stabmin}
Assume in addition of the hypotheses of Theorem \ref{thstab} that  $x\mapsto\hat T(x)$ achieved a strict local minimum at $\hat x_0$. 
Then, up to choosing $\hat \epsilon_0$ smaller, 
the function  $x\mapsto T(x)$ achieves a local minimum in the open ball $B(\hat x_0, \hat \epsilon_0)$ with $\nabla T(x_0)=0$ and $w_{x_0}(s) \to \pm\kappa_0$ as $s\to \infty$.\\ 
 The same statement holds for a local maximum  blow-up time. 
\end{coro}

\medskip

Our results rely on two features:\\
- a good understanding of the dynamics of equation \eqref{eqw}, near the set $\{\pm \kappa(d)\}$; this has already been done in \cite{MZtds}, and we recall that result in Proposition \ref{cortrap} below;\\
-  the following rigidity theorem for solutions to equation \eqref{eqw}:
\begin{thm}[A rigidity theorem for equation \eqref{eqw}] \label{thrig} Consider $w(y,s)$ a solution of equation \eqref{eqw} defined for all $|y|<A^*$ and $s\in \m R$ for some $A^*>1$ and satisfying for all $s\in \m R$,
\begin{equation}
\|(w(s), \partial_s w(s))\|_{H^1\times L^2 (|y|<A^*)}\le M^* \mbox{ and }\forall |y|<1,\;\;w(y,s) = e^*\kappa(d^*,y),\label{hypw}
\end{equation}
for some $e^*=\pm 1$, $|d^*|<1$ and $M^*>0$.
Then,
\begin{equation}\label{sol}
\forall |y|<A^*\mbox{ and }s\in \m R,\;\;
w(y,s) = e^*\kappa(d^*,y).
\end{equation}
\end{thm}
Using the similarity variables' formulation \eqref{defw}, we get 
 the following rigidity theorem for ancient solutions of equation \eqref{equ} defined in a {\it non}-characteristic cone and close to the set $\{\pm \kappa(d)\}$ as $s\to -\infty$: 

\medskip

\noindent {\bf Theorem \ref{thrig}'}
(A rigidity theorem for equation \eqref{equ}){\it Consider $u(x,t)$ a solution of equation \eqref{equ} defined for all $(x,t) \in \q C_{x^*, T^*, \delta^*}$ introduced in \eqref{defcone}, which satisfies for all $s\in \m R$:
\begin{equation}
\|(w_{x^*,T^*}(s), \partial_s w_{x^*, T^*}(s))\|_{H^1\times L^2 (|y|<\frac 1{\delta^*})}\le M^*
\mbox{ and }
\forall |y|<1,\;\;w_{x^*,T^*}(y,s) = e^*\kappa(d^*,y),\label{hyp'}
\end{equation}
where $w_{x^*,T^*}$ is defined in \eqref{defw}, for some $x^*\in \m R^N$, $T^*\in \m R$, $\delta^*\in(0,1)$, $e^*=\pm 1$, $|d^*|<1$ and $M^*>0$. Then
for all $(x,t) \in C_{x^*, T^*, \delta^*}$,
\begin{equation*}
u(x,t) = e^*\kappa_0 \frac{(1-|d^*|^2)^{\frac 1{p-1}}}{(T^*-t +d^*.(x-x^*))^{\frac 2{p-1}}}.
\end{equation*}
}
\begin{nb} 
 From the conclusion, we see that $u$
can be extended to a set larger than $C_{x^*, T^*, \delta^*}$, namely the half-space $\{(x,t)\;|\;T^*-t +d^*.(x-x^*)>0\}$.
\end{nb}

Compared to the analysis of the one-dimensional case treated in \cite{MZcmp08}, the main difficulties in the adaptation lay in the proof of the rigidity result (Theorem \ref{thrig} above):\\
- we need a very good understanding of the linearized operator of equation \eqref{eqw} near the stationary solution $\kappa(d)$ \eqref{defkd} (see Proposition \ref{propexpo} below); such an operator was studied in details in the companion paper \cite{MZtds}, showing $N-1$ new degenerate directions coming from the derivatives of $\kappa(d)$ with respect to the angular directions of $d$;\\
- we also need to know properties of the wave operator, and such properties depend on the dimension (e.g. Strichartz estimates); this is mainly done in Appendix \ref{applemliouv}.\\
Note that the derivation of Theorem \ref{threg} needs less adaptations of the one-dimensional case. As for Theorem \ref{thstab}, some steps are similar to the proof of Theorem \ref{threg}, and others are original. 
In any case, in the whole paper, we only sketch the steps where some adaptation is needed, and refer to \cite{MZcmp08} for those where the adaptation is straightforward (see Section \ref{secreg} below).

\bigskip

We proceed in four sections:\\
- In Section \ref{secdyn}, we recall from \cite{MZtds} some results about the behavior of solutions of equation \eqref{eqw} near the set $\{\pm \kappa(d)\}$;\\
- In Section \ref{secreg}, we assume Theorem \ref{thrig} and prove Theorem \ref{threg};\\
- In Section \ref{secstab}, we assume Theorem \ref{thrig} and prove Theorem \ref{thstab} together with Corollary \ref{stabmin};\\
- In Section \ref{secrig}, we prove Theorem \ref{thrig} and Theorem \ref{thrig}'.\\
{\bf Acknowledgement}. The authors wish to thank the referee for his valuable comments which undoubtedly improved the presentation of the results.
\section{Behavior of solutions of equation \eqref{eqw} near the set $\{\pm \kappa(d)\}$}\label{secdyn}
A key step towards Theorem \ref{threg} is to understand the dynamics of equation \eqref{eqw} for initial data near the set $\{\pm\kappa(d)\;|\;|d|<1\}$. This was done in \cite{MZtds} (see Theorem 1 in that paper): 
\begin{prop}[Behavior of solutions of equation \eqref{eqw} near $\pm \kappa(d,y)$; see \cite{MZtds}]\label{cortrap} 
There exist $\eb>0$, $K_0>0$ and $\mu_0>0$ such that 
for any $\bar x\in \RR$, $\bar s\ge -\log T(\bar x)$, $\bar\omega=\pm 1$ and $|\bar d|<1$, if
\begin{equation*}
\ee\equiv\left \|\vc{w_{\bar x}(\bar s)}{\partial_s w_{\bar x}(\bar s)}-\bar\omega\vc{\kappa(\bar d)}{0}\right\|_{\H}\le \eb,
\end{equation*}
then,
there exists $|d_\infty(\bar x)|<1$ such that
\begin{equation}\label{conv00}
\forall s\ge \bar s,\;\;\|(w_{\bar x}(s), \partial_s w_{\bar x}(s))- (\kappa(d_\infty),0)\|_{\q H}\le K_0\ee e^{-\mu_0(s-\bar s)}
\end{equation}
and 
\begin{equation*}
|\arg\tanh |\bar d|-\arg\tanh |d_\infty||+\frac{|\bar d-d_\infty|}{\sqrt{1-|\bar d|}}\le K_0\ee.
\end{equation*}
\end{prop}
We also need to recall from \cite{MZtds} an exponential convergence property for solutions of equation \eqref{eqw} near $\{\pm \kappa(d)\}$. In order to do so, we need to introduce $\kappa^*(d,\mu e^s,y)=(\kappa_1^*, \kappa_2^*)(d,\mu e^s,y)$, given for all $|d|<1$ and $\nu >-1+|d|$ by 
\begin{equation}\label{defk*}
\kappa_1^*(d,\nu, y) =
\ds\kappa_0\frac{(1-|d|^2)^{\frac 1{p-1}}}{(1+d.y+\nu)^{\frac 2{p-1}}},\;
\kappa_2^*(d,\nu, y) = \nu \pnu \kappa_1^*(d,\nu, y) =
\ds-\frac{2\kappa_0\nu}{p-1}\frac{(1-|d|^2)^{\frac 1{p-1}}}{(1+d.y+\nu)^{\frac {p+1}{p-1}}}.
\end{equation}
Note that for any $\mu\in \m R$, the function $\kappa_1^*(d,\mu e^s, y)$ is an explicit solution of equation \eqref{eqw}.\\
In Proposition 3.9 of \cite{MZtds}, we proved the following:
\begin{prop}[Exponential convergence of solutions of \eqref{eqw} near $\{\kappa^*(d,\nu)\}$; see \cite{MZtds}] \label{propexpo} There exists $\delta_2>0$ such that for any $s_1\ge s_0$, if $w\in C([s_0,s_1],\q H)$ is a solution of equation \eqref{eqw} satisfying
\begin{equation}\label{proche0}
\forall s\in [s_1,s_2],\;\;\left\|\vc{w(s)}{\partial_s w(s)}-\bar\omega\vc{\kappa(\bar d(s))}{0}\right\|_{\q H}\le \delta_2,
\mbox{ with }|\bar d(s)|<1\mbox{ and }\bar \omega=\pm 1, 
\end{equation}
then, there exist $C^1$ parameters $|d(s)|<1$ and $\nu(s) >-1+|d(s)|$ such that
\[
\forall s\in [s_0,s_1],\;\;
\|q(s)\|_{\q H}\le \frac{e^{-\delta_2(s-s_0)}}{\delta_2}\|q(s_0)\|_{\q H}
\]
where $q(y,s) =(w(y,s),\partial_s w(y,s)) - \bar \omega \kappa^*(d(s), \nu(s))$ and $\kappa^*$ is defined in \eqref{defk*}.
\end{prop}

\section{Openness of $\RR_0$ and $C^1$ regularity of the blow-up surface on $\RR_0$}\label{secreg}
In this section, we assume the rigidity result of Theorem \ref{thrig} and prove Theorem \ref{threg}.
Like for the proof of Theorem 1 page 58 in \cite{MZcmp08}, we proceed in two subsections:

- we first consider $x_0\in \RR_0$ and show that $T(x)$ is differentiable at $x_0$ with $\nabla T(x_0) = d(x_0)$ defined in \eqref{defr0};

- then, we give the proof of Theorem \ref{threg}.

\subsection{Differentiability of the blow-up surface on $\RR_0$}
We prove here an analogous statement to Proposition 2.1 page 60 in \cite{MZcmp08}:
\begin{prop}[Differentiability of the blow-up surface on $\RR_0$]\label{propdiff} If $x_0\in \RR_0$, then $x\mapsto T(x)$ is differentiable at $x=x_0$ and 
\begin{equation}\label{diff}
\nabla T(x_0)=d(x_0)
\end{equation}
 introduced in \eqref{defr0}.
\end{prop}
\begin{proof}
The proof follows the same pattern as the one-dimensional case, except for two ingredients:\\
- our rigidity theorem \ref{thrig} replaces the one-dimensional Liouville Theorem of \cite{MZcmp08} (note that our theorem \ref{thrig} has an extra assumption on the data);\\
- the continuity of solutions to equation \eqref{eqw} in the $H^1\times L^2$ weak topology has to be checked, since both the fundamental wave operator and Sobolev embeddings change in higher dimensions (the one dimensional proof doesn't apply directly, so we had to recollect results from previous literature in Appendix \ref{applemliouv} to show that continuity).\\
  For this reason, we mention here the only step where the theorem \ref{thrig} and the continuity of solutions to \eqref{eqw} are used, and refer the reader to the proof of Proposition 2.1 page 60 in \cite{MZcmp08}.

\medskip

We consider $x_0\in\RR_0$. From translation invariance of equation \eqref{equ}, we may assume that 
\[
x_0=T(x_0)=0.
\]
 Since $\RR_0\subset \RR$, we see from \eqref{nonchar} that
\begin{equation}\label{wadi}
\q C_{x_0,T(x_0),\delta_0}\cap \{t\ge t_0\} \subset D
\end{equation}
for some $\delta_0\in (0,1)$ and $t_0<T(x_0)$.\\ 
Furthermore, by definition \eqref{defr0} of $\RR_0$, we see that \eqref{profileN} holds for some $d(0)\in B(0,1)$ and $e(0)=\pm 1$. Up to replacing $u(x,t)$ by $-u(x,t)$ (also solution to equation \eqref{equ}), we assume that $e(0)=1$, hence, from \eqref{profileN1}, we see that
\begin{equation}\label{profileN1}
\vc{w_0(s)}{\partial_s w_0(s)}\to \vc{\kappa(d(0))}{0}
\mbox{ in }\q H
\mbox{ as }s\to \infty,
\end{equation}
As we mentioned in the beginning of this proof, the only delicate point in the adaptation of the one-dimensional case is the following lemma where we extend the convergence in \eqref{profileN} to a large set (with no weights), as we recall in the following statement, analogous the Lemma 2.2 page 61 in \cite{MZcmp08}:
\begin{lem}[Convergence in selfsimilar variables on larger sets] \label{proplarger} For all $\delta_0'\in (\delta_0,1)$, it holds that
\[
\left\|\vc{w_{0}(s)}{\partial_s w_{0}(s)}-\vc{\kappa(d(0))}{0}\right\|_{H^1\times L^2 \left(|y|<\frac 1{\delta_0'}\right)} \to 0\mbox{ as } s\to \infty.
\]
\end{lem}
\begin{proof} Consider some $\delta_0'\in (\delta_0,1)$. The beginning of the proof is the same as the one-dimensional case (Lemma 2.2 page 61 in \cite{MZcmp08}).\\
For simplicity, we denote $w_{0}$ by $w$. Using the uniform bound on the solution at blow-up (Theorem 2' in \cite{MZimrn05}) and the covering technique in that paper (Proposition 3.3 in \cite{MZimrn05}), we get for all $s\ge -\log T(0) +1$, 
\begin{equation}\label{petit1}
\left\|\vc{w(s)}{\partial_s w(s)}\right\|_{H^1\times L^2 \left(|y|<\frac 1{\delta_0'}\right)} \le K
\end{equation}
for some constant $K$.\\
We proceed by contradiction and assume that for some $\epsilon_0>0$ and some sequence $s_n\to \infty$, we have
\begin{equation}\label{contra}
\forall n\in \m N,\; \left\|\vc{w(s_n)}{\partial_s w(s_n)}-\vc{\kappa(d(0)}{0}\right\|_{H^1\times L^2 \left(|y|<\frac 1{\delta_0'}\right)} \ge \epsilon_0>0.
\end{equation}
Let us introduce the sequence
\begin{equation}\label{defwn}
w_n(y,s)=w(y,s+s_n).
\end{equation}
Using the uniform bound stated in \eqref{petit1}, we can assume that 
\begin{equation}\label{convk0}
w_n(0)\wto z_0 \mbox{ in }H^1\left(|y|<\frac 1{\delta_0'}\right)
\mbox{ and }
\partial_s w_n(0)\wto v_0 \mbox{ in }L^2\left(|y|<\frac 1{\delta_0'}\right)
\end{equation}
as $n\to \infty$ for some $(z_0, v_0)\in H^1\times L^2 \left(|y|<\frac 1{\delta_0'}\right)$.
Since we have from \eqref{profileN1}, the definitions \eqref{defnh} and \eqref{defwn} of the norm in $\H$ and $w_n$,
\[
\left\|\vc{w_n(s)}{\partial_s w_n(s)}-\vc{\kappa(d(0))}{0}\right\|_{H^1\times L^2(|y|<1-\epsilon)}\to 0 \mbox{ as }n\to \infty
\]
for any $s \in \m R$ and $\epsilon\in (0,1)$, we deduce from \eqref{convk0} that 
\begin{equation}\label{id0}
\forall |y|<1,\;\; z_0(y)=\kappa(d(0),y)\mbox{ and }v_0(y)=0
\end{equation}
(note that we still need to determine $(z_0, v_0)$ for $1<|y|<\frac 1{\delta_0'}$).
The following claim allows us to conclude, thanks to the rigidity Theorem \ref{thrig} (and here start the novelties with respect to the one-dimensional case):
\begin{lem}[Existence of a limiting object]\label{lemliouv} 
There exists $W(y,s)$ a solution to \eqref{eqw} defined for all $|y|<\frac 1{\delta_0'}$ and $s\in \m R$ such that:\\
(i) $W(0,y)=z_0(y)$ and $\partial_s W(0,y) = v_0(y)$ for all $|y|<\frac 1{\delta_0'}$ and (up to extracting a subsequence still denoted by $w_n$), the convergence is strong in \eqref{convk0}.\\
(ii) For all $s\in \m R$ and $|y|<1$, $W(y,s) = \kappa(d,y)$.\\
(iii) For all $s\in \m R$,
\begin{equation}\label{boundkk}
\left\|\vc{W(s)}{\partial_s W(s)}\right\|_{H^1\times L^2\left(|y|<\frac 1{\delta_0'}\right)}\le K
\end{equation}
where $K$ is defined in \eqref{petit1}.
\end{lem}
\begin{proof} This proof of this statement uses the fundamental solution of the free wave operator, which depends on the dimension. Thus, the one-dimensional case proof doesn't hold here, which makes this claim a novelty of our argument (see Claim 2.3 page 62 in \cite{MZcmp08} for the one-dimensional case). Since the proof is mostly technical, we leave it to Appendix \ref{applemliouv}.
\end{proof}

\medskip

Indeed, from this claim, we see that $W(y,s)$ satisfies the hypothesis of Theorem \ref{thrig}. Therefore, either $W\equiv 0$ or there exists $\mu_0\ge 0$, $d_0\in B(0,1)$ and $\theta_0=\pm 1$ such that:
\begin{equation}\label{27b}
\forall |y|<\frac 1{\delta_0'}\mbox{ and }s\in \m R,\;\;W(y,s) = \theta_0\kappa_0\frac{(1-|d_0|^2)^{\frac 1{p-1}}}{(1+\mu_0 e^s + d_0. y)^{\frac 2{p-1}}}
\end{equation}
on the one hand, where $\kappa_0$ is defined in \eqref{defkd}. On the other hand, using \eqref{id0}, (i) of Lemma \ref{lemliouv}, and the definition \eqref{defkd} of $\kappa(d,y)$, we see that
\begin{equation}\label{27c}
\forall y\in(-1,1),\;\;W(y,0)=z_0(y)=\kappa(d(0),y)=\kappa_0\frac{(1-|d(0)|^2)^{\frac 1{p-1}}}{(1+d(0).y)^{\frac 2{p-1}}}.
\end{equation}
Comparing \eqref{27b} and \eqref{27c} when $y\in(-1,1)$ and $s=0$, we see that $\theta_0=1$, $d_0=d(0)$ and $\mu_0=0$, hence, from \eqref{27b},
\[
\forall |y|<\frac 1{\delta_0'}\mbox{ and }s\in \m R,\;\;W(y,s) = \kappa(d(0),y).
\]
In particular, from \eqref{defwn}, \eqref{convk0} and (i) of Lemma \ref{lemliouv}, this implies that (up to extracting a subsequence still denoted by $w_n$),
\[
\left\|\vc{w(s_n)}{\partial_s w(s_n)}-\vc{\kappa(d(0))}{0}\right\|_{H^1\times L^2 \left(|y|<\frac 1{\delta_0'}\right)}\to 0\mbox{ as }n\to \infty,
\]
which contradicts \eqref{contra}. Thus, Lemma \ref{proplarger} holds.
\end{proof}

\bigskip

With Lemma \ref{proplarger}, one can see that the contradiction argument in Step 2 page 62 in \cite{MZcmp08} extends straightforwardly to the higher dimensional case, and shows that
\[
\forall i =1,\dots,N,\;\;\partial_{x_i}T(0)=d_i(0),
\]
concluding therefore the proof of Proposition \ref{propdiff}.
\end{proof}

\subsection{Proof of Theorem \ref{threg}}\label{sub3}

We prove Theorem \ref{threg} here.

\begin{proof}[Proof of Theorem \ref{threg}]
Let $x_0\in \RR_0$. One can assume that $x_0=T(x_0)=0$ from translation invariance. From the definition \eqref{defr0} and Proposition \ref{propdiff},  we know (up to replacing $u(x,t)$ by $-u(x,t)$) that \eqref{profileN} holds with some $d(0)\in B(0,1)$ and $e(0)=1$, and that $T(x)$ is differentiable at $0$ with 
\begin{equation}\label{assumption}
\nabla T(0)=d(0).
\end{equation}
We aim at showing that all the points in some neighborhood of $0$ are in $\RR_0$ and that the function $x\mapsto T(x)$ is of class $C^1$ on that neighborhood.

\bigskip

We proceed in 3 steps. 

- In Step 1, we show that for some small $\delta_1>0$, $\RR\cap \{|x|<\delta_1\} \subset \RR_0$. More precisely, if $x\in\RR$ and $|x|<\delta_1$, then \eqref{profileN} holds for $w_x$, for some $|d(x)|<1$ and $e(x)=1$ with $d(x)\to \dd$ as $x\to 0$.

- In Step 2, using a geometrical construction and the previous step, we show that for some $\delta_2>0$, $\{|x|<\delta_2\} \subset \RR$. 

- In Step 3, using Steps 1 and 2, we conclude the proof of Theorem \ref{threg}. 

\bigskip

{\bf Step 1: $\RR\cap \{|x|<\delta_1\} \subset \RR_0$, for some $\delta_1>0$}

Using the dynamical study in selfsimilar variables \eqref{defw}, we claim the following:
\begin{lem}[Convergence in selfsimilar variables for $x$ close to $0$]\label{lemp1} 
For all $\epsilon>0$, there exists $\eta>0$ such that if $|x|\le \eta$ and $x\in\RR$, then, $x\in \RR_0$. More precisely, \eqref{profileN} holds for $w_{x}$ with $|d(x)-\dd|\le \epsilon$ and $e(x)=1$.
\end{lem}
\begin{nb} Here, we don't assume that all the points in some neighborhood of $0$ are uniformly non characteristic (that is, $\delta_0(x)$ defined in \eqref{nonchar} may have no upper bound strictly lower than $1$ in any neighborhood of $0$). We use instead the fact that we completely understand the dynamical structure of equation \eqref{eqw} in $\H$ close to the stationary solution $\kappa(d(0),y)$.
\end{nb}
\begin{proof}$ $\\
- Since $0$ is non characteristic, we have from \eqref{petit1}, for all $s\ge s_1$ for some $s_1\in \m R$, $\left\|(w_0(s),\partial_s w_0(s))\right\|_{H^1\times L^2 \left(|y|<\frac 1{\delta_0'}\right)} \le K$
for some constant $K$, where $\delta_0'\in (\delta_0, 1)$ is fixed.
Again from the fact that $0\in\RR_0$, as we noted earlier, we know that \eqref{profileN} holds, hence $(w_0(s), \partial_s w_0(s))$ converges to $(\kappa(\dd,.),0)$ as $s\to \infty$ in the norm of $\H$.

\medskip

\noindent - Since for fixed $s$, we have $(w_x(y,s), \partial_s w_x(y,s))\to (w_0(y,s), \partial_s w_0(y,s))$ in $\H$ from the continuity of solutions to equation \eqref{eqw} with respect to initial data, we know that for all $\epsilon>0$, there exists $s_0(\epsilon)\ge s_1$ and $\eta(\epsilon)>0$ such that for all $|x|<\eta(\epsilon)$, 
\[
\left\|\vc{w_{x}(\cdot,s_0(\epsilon))}{\partial_s w_{x}(\cdot,s_0(\epsilon))} - \vc{\kappa(\dd,.)}{0}\right\|_{\H} \le \epsilon.
\]
- From Proposition \ref{cortrap}, for a small enough fixed $\epsilon>0$, we have that for all $x\in \RR$ with $|x|<\eta(\epsilon)$, there exists $d(x)$ such that 
\[
\left\|\vc{w_x(y,s)}{\partial_s w_x(y,s)}-\vc{\kappa(d(x),y)}{0}\right\|_{\H} \to 0\mbox{ as } s\to \infty
\]
and 
\[
|d(x)-\dd|\le C \epsilon,
\]
hence $x\in \RR_0$.
This concludes the proof of Lemma \ref{lemp1}.
\end{proof}

{\bf Step 2: The Lipschitz constant of $T(x)$ around $0$ is less than  $(1+|\dd|)/2$}

Fix $\epsilon_0$ small enough such that 
\begin{equation}\label{defe0}
0<\epsilon_0\le \frac {1-|\dd|}{40}.
\end{equation}
 Using \eqref{assumption} and Lemma \ref{lemp1}, we see that there exists $\eta_0>0$ such that
\begin{equation}\label{lip0}
\forall |x|\le \eta_0,\;\;|T(x)-T(0)-\dd\cdot x|\le \epsilon_0|x|,
\end{equation}
and 
\begin{equation}\label{rr0}
\mbox{if in addition, }x\in\RR,\mbox{ then, }x\in\RR_0,
\end{equation}
 in particular \eqref{profileN} holds for $w_{x}$ with 
\begin{equation}\label{step2}
|d(x)-\dd|\le \epsilon_0
\end{equation}
and $e(x)=1$.
We now claim the following:
\begin{lem}\label{proplip1/2}{\bf (The Lipschitz constant of $T(x)$ around $0$ is less than $(1+|\dd|)/2$)} If $|x|\le \frac{\eta_0}{10}$\mbox{ and }$|y|\le \frac{\eta_0}{10}$, then
\[
 |T(x)-T(y)|\le \frac {1+|\dd|}2 |x-y|.
\] 
\end{lem}
\begin{proof}
We proceed by contradiction, and assume that for some
\begin{equation}\label{x0}
|x_0|\le \frac{\eta_0}{10}
\end{equation}
 and $|y_0|\le \frac{\eta_0}{10}$, we have $|T(x_0)-T(y_0)|> \frac {1+|\dd|}2 |x_0-y_0|$. Note in particular that $x_0\neq y_0$.
Up to renaming $x_0$ and $y_0$, we may assume that $T(x_0)> T(y_0)$, so that 
\begin{equation}\label{absurde0}
T(x_0)-T(y_0)> \frac {1+|\dd|}2 |x_0-y_0|>0.
\end{equation}
Keeping this $x_0$, we may change $y_0$ so that we maximize the distance $|x_0-y_0|$ in the set of all $|y_0|\le \eta_0$ satisfying \eqref{absurde0}. This way, we see that
\begin{equation}\label{y1min0}
\mbox{if }|y|\le \eta_0\mbox{ and }|y-x_0|\ge |y_0-x_0|,\mbox{ then }T(x_0)-T(y)\le \frac {1+|\dd|}2 |x_0-y|.
\end{equation}
We claim that 
\begin{equation}\label{y0inf0}
|y_0|\le \frac{\eta_0}5.
\end{equation}
Indeed, note first from \eqref{absurde0} and \eqref{lip0} that 
\begin{align*}
|x_0-y_0|&\le \frac 2{1+|\dd|}(T(x_0)-T(y_0))\\
&\le \frac 2{1+|\dd|}\left[(T(x_0)-T(0))- (T(y_0)-T(0))\right]\\
&\le  \frac {2\dd}{1+|\dd|}\cdot (x_0-y_0)+ \frac {2\epsilon_0}{1+|\dd|}\left(|x_0|+|y_0|\right)\\
&\le  \frac {2|\dd|}{1+|\dd|}|x_0-y_0|+ \frac {2\epsilon_0}{1+|\dd|}\left(|x_0|+|y_0|\right).
\end{align*}
Since $\frac {2|\dd|}{1+|\dd|}<1$, $|x_0|\le \frac{\eta_0}{10}$ and $|y_0|\le \eta_0$,  this yields
\begin{equation*}
|x_0-y_0|\le \frac {2\epsilon_0}{1-|\dd|}\left(|y_0|+|x_0|\right)
\le \frac {4\epsilon_0\eta_0}{1-|\dd|}\le \frac{\eta_0}{10}
\end{equation*}
thanks to the smallness condition \eqref{defe0} on $\epsilon_0$, hence \eqref{y0inf0} follows.  Thus, by minimality in \eqref{y1min0}, we see that
\begin{equation}\label{construction1}
|x_0|\le \frac{\eta_0}{10},\;\;
|y_0|\le \frac{\eta_0}5,\;\;
T(x_0)-T(y_0)= \frac {1+|\dd|}2 |x_0-y_0|>0,
\end{equation}
and 
\begin{equation}\label{minbord}
T(y_0)=\min_{|y-x_0|=|y_0-x_0|}T(y).
\end{equation}
Considering a family of cones ${\cal C}_{x_0,t_0, |\dd|+2\epsilon_0}$ where $t_0\in \m R$, we may select the largest $t_0$ such that the graph of $y\mapsto T(y)$ for $|y-x_0|\le |x_0-y_0|$ lays above the cone, in the sense that
\begin{equation}\label{construction2}
\mbox{if }|y-x_0|\le |y_0-x_0|,\mbox{ then }T(y) \ge t_0 -\left(|\dd|+2\epsilon_0\right)|y-x_0|.
\end{equation}
By maximality of $t_0$, there is $\bar y$ such that
\begin{equation}\label{construction2bis}
|\bar y-x_0|\le |y_0-x_0|\mbox{ and }T(\bar y)= t_0 -(|\dd|+2\epsilon_0)|\bar y-x_0|.
\end{equation}
 We claim the following:
\begin{cl}[The point $\bar y\neq x_0$ and $\bar y\in \RR$]\label{clnonchar200}
We have, $|\bar y|\le \frac {2\eta_0}5$, $\bar y\neq x_0$ and there exists $\eta_1>0$ such that
\begin{equation}\label{mnonchar2000}
 \mbox{if }|x-\bar y|\le \eta_1,\mbox{ then }T(x) \ge T(\bar y)-\frac {1+|\dd|}2 |x-\bar y|.
\end{equation}
In particular, $\bar y\in \RR$.
\end{cl}
\begin{proof} We don't prove the last line of the claim, since it follows from \eqref{mnonchar2000} by definition \eqref{nonchar} of the notion of non-characteristic point.\\
Using \eqref{x0} and \eqref{y0inf0}, we see that $|\bar y|\le |\bar y-x_0|+|x_0|=|y_0-x_0|+|x_0|\le |y_0|+2|x_0|\le \frac{2\eta_0}5$.\\
Note then from \eqref{construction2bis}, \eqref{construction2} and a triangular inequality, that 
\begin{equation}\label{ybarcone}
\mbox{if }|y-x_0|\le |y_0-x_0|,\mbox{ then }T(y)\ge T(\bar y)-\left(|\dd|+2\epsilon_0\right)|y-\bar y|.
\end{equation}
Taking $y=y_0$ in this inequality, we see from \eqref{construction1} and the smallness of $\epsilon_0$ \eqref{defe0} that $\bar y \neq x_0$.\\
Now, if $|\bar y -x_0|<|y_0-x_0|$, then there exists $\eta_1'>0$ small enough such that if $|y-\bar y|\le \eta_1'$, then $|y-x_0|\le |y_0-x_0|$, and \eqref{mnonchar2000} follows from \eqref{ybarcone}.\\
If $|\bar y -x_0|=|y_0-x_0|$, then we see from \eqref{construction2bis} and \eqref{construction2} that $T(y_0)\ge T(\bar y)$, hence from \eqref{minbord}, we see that 
\[
T(\bar y)=T(y_0)\mbox{ and }|\bar y -x_0|=|y_0-x_0|.
\]
Using \eqref{y1min0} and a triangular inequality, we see that
\begin{equation}\label{youtside}
\mbox{if }|y|\le \eta_0\mbox{ and }|y-x_0|\ge |y_0-x_0|,\mbox{ then }T(y) \ge T(\bar y) - \frac{1+|\dd|}2|y-\bar y|.
\end{equation}
Since $|\bar y|\le \frac{2\eta_0}5$,
it follows that when $|y-\bar y|\le \frac{3\eta_0}5$, we have $|y|\le \eta_0$.
Introducing $\eta_1=\min(\eta_1', \frac{3\eta_0}5)$ and 
according to whether $|y-x_0|\le |y_0-x_0|$ or not, we may use \eqref{ybarcone} or \eqref{youtside} and the smallness of $\epsilon_0$ \eqref{defe0} to conclude the proof of Claim \ref{clnonchar200}.
\end{proof}

\bigskip

In Claim \ref{clnonchar200}, we have just proved that $\bar y\in \RR$ and $|\bar y|\le \frac{2\eta_0}5$.
Using \eqref{step2} and Proposition \ref{propdiff}, we see that $\bar y \in \RR_0$,
$T(x)$ is differentiable at $x=\bar y$ and 
\[
|\nabla T(\bar y)-\dd|= |d(\bar y)-\dd|\le \epsilon_0
\]
 on the one hand. On the other hand, from \eqref{construction2}, \eqref{construction2bis} and the fact that $\bar y\neq x_0$, we have $|\nabla T(\bar y)|\ge |\nabla T(\bar y)\cdot \frac {\bar y-x_0}{|\bar y -x_0|}|\ge |\dd|+2\epsilon_0$, 
which leads to a contradiction. This concludes the proof of Lemma \ref{proplip1/2}.
\end{proof}

{\bf Step 3: Conclusion of the proof}

Using Lemma \ref{proplip1/2}, we see that for all $|x|\le \frac{\eta_0}{20}$, $x\in\RR$. Using \eqref{rr0}, we see that $x\in \RR_0$. Using Proposition \ref{propdiff}, we see that $T$ is differentiable at $x$ and $\nabla T(x)= d(x)$ where $d(x)$ is such that \eqref{profileN} holds for $w_x$. Using Lemma \ref{lemp1}, we see from \eqref{assumption} that $\nabla T(x)=d(x)\to d(0)=\nabla T(0)$ as $x\to 0$ and $e(x)=1$. This concludes the proof of Theorem \ref{threg}.
\end{proof}

\section{Stability results related to non-characteristic points in $\RR_0$}\label{secstab}
In this section, we assume again Theorem \ref{thrig} and prove Theorem \ref{thstab} together with Corollary \ref{stabmin}, each in a separate subsection. Since we have already shown that Theorem \ref{threg} follows from Theorem \ref{thrig}, we will use Theorem \ref{threg} and all the statements of Section \ref{secreg} in our argument.
\subsection{Stability of the notion of a non-characteristic blow-up point in $\RR_0$}\label{secthstab}

We prove Theorem \ref{thstab} here, assuming that Theorem \ref{thrig} holds.

\medskip

\begin{proof}[Proof of Theorem \ref{thstab} assuming that Theorem \ref{thrig} holds]
 Consider $\hat u(x,t)$ a blow-up solution of equation \eqref{equ} with initial data $(\hat u_0, \hat u_1)\in H^1\times L^2$, and $\hat x_0\in \hat \RR_0$. From Theorem \ref{threg}, there is $\eta_1>0$ such that $B(\hat x_0, \eta_1)\subset \hat \RR_0$ and the function $x\mapsto \hat T(x)$ is of class $C^1$ in that ball. 
We will prove that for some $\hat \epsilon_0\in (0, \eta_1]$ and for all $(u_0,u_1)$ such that $\|(u_0,u_1)-(\hat u_0, \hat u_1)\|_{H^1\times L^2}\le \hat \epsilon_0$, 
the solution $u(x,t)$ of equation \eqref{equ} with initial data $(u_0,u_1)$ blows up in finite time with the ball $B(\hat x_0, \hat \epsilon_0)\subset \RR_0$.

\medskip

Since $\hat x_0\in \hat \RR_0$, from sign, translation and scaling invariance of equation \eqref{equ}, together with Lemma \ref{proplarger}, we claim that
\begin{equation}\label{init}
\hat x_0=0,\;\;\hat T(\hat x_0)=1\mbox{ and }(\hat w_0(s), \partial_s \hat w_0(s)) \to (\kappa(\hat d(0)),0)\mbox{ in }H^1\times L^2(|x|<A),
\end{equation}
as $s\to \infty$, for some $|\hat d(0)|<1$ and 
any $A$ in the interval $[1,\frac 1{|\hat d(0)|})$ (which becomes the interval $[1,+\infty)$ if $\hat d(0)=0$). The only delicate point in \eqref{init} lays in the justification of the validity interval for $A$. In order to do so, we first note from Theorem \ref{threg} that the blow-up surface $x\mapsto \hat T(x)$ is of class $C^1$ in the ball $B(0, \eta_1)$ with $\hat d(0)$ as slope. Therefore, for any $A\in [1,\frac 1{|\hat d(0)|})$, there is some time $t_0(A) \in [0,1)$ such that the cone $\q C_{0,1,A}\cap \{t\ge t_0(A)\}\subset D_u$. Shifting the time origin to $t_0(A)$ and applying Lemma \ref{proplarger}, we get \eqref{init}.

\medskip

We proceed in 3 steps:\\
- In Step 1, we show the continuity of the blow-up time with respect to initial data;\\
- In Step 2, we prove the differentiability of $x\mapsto T(x)$ at non-characteristic points $x$ near $0$ and for initial data $(u_0,u_1)$ close enough to $(\hat u_0, \hat u_1)$;\\
- In Step 3, we proceed by contradiction to conclude the proof.

\bigskip

{\bf Step 1: Continuity of the blow-up time with respect to space and to initial data}

We claim the following:
\begin{lem}[Continuity of the blow-up time with respect to initial data]\label{lemcont}
For any $\delta>0$, there exists $\epsilon_1>0$ and $\eta_1>0$ such that 
\[
|T(x)-1|\le 2\delta,
\]
whenever $\|(u_0, u_1)-(\hat u_0, \hat u_1)\|_{H^1\times L^2} \le \epsilon_1$ and $|x|\le \eta_1$, where $x\mapsto T(x)$ is the blow-up surface of $u(x,t)$, the solution of equation \eqref{equ} with initial data $(u_0,u_1)$.
\end{lem}
\begin{nb} This is a twin statement of our analogous result for the semilinear heat equation proved with Fermanian in \cite{FKMZma00} (see Lemma 1.5 page 354 in that paper). As a matter of fact, both proofs follow the same pattern. 
\end{nb}
\begin{nb}Of course, from the finite speed of propagation, we may change the $H^1\times L^2$ norm by a the $H^1\times L^2(|x|<1+\alpha)$ for any $\alpha>0$. 
\end{nb}
\begin{proof} Let us first briefly explain the proof, before giving details. The proof follows the pattern we developed for the continuity of the blow-up time in the case of the semilinear heat equation in \cite{FKMZma00} (see Lemma 1.5 page 354 in that paper):\\
- the lower semicontinuity follows from the Cauchy theory;\\
- the upper semicontinuity follows from the knowledge of the blow-up behavior \eqref{init}, the similarity variables' transformation \eqref{defw} together with a blow-up criterion related to the following Lyapunov functional for equation \eqref{eqw}  (see Antonini and Merle \cite{AMimrn01} and Lemma \ref{lemcrit} below for a statement):
\begin{equation}\label{defenergy}
E(w(s))= \iint \left(\frac 12 \left(\partial_s w\right)^2 + \frac 12 |\nabla w|^2 -\frac 12 (y\cdot\nabla w)^2+\frac{(p+1)}{(p-1)^2}w^2 - \frac 1{p+1} |w|^{p+1}\right)\rho dy.
\end{equation}

\medskip

Let us now give the proof of Lemma \ref{lemcont} in details.\\
Consider $A\in (1, \frac 1{|\hat d(0)|})$ to be fixed close enough to $\frac 1{|\hat d(0)|}$ and introduce $A'(A)=A-\left(\frac 1{|\hat d(0)|}-A\right)$ and $A''(A)=A-2\left(\frac 1{|\hat d(0)|}-A\right)$ (if $\hat d(0)=0$, we take $A$ large enough and introduce $A'(A)=A-1$ and $A"(A)=A-2$). Note in particular that
\[
1<A"<A'<A<\frac 1{|\hat d(0)|}\mbox{ and }A', A"\to \frac 1{|\hat d(0)|}\mbox{ as }A \to \frac 1{|\hat d(0)|}.
\]
Consider also $\epsilon>0$ to be fixed small enough later. From \eqref{init}, there exists $\delta_0(A,\epsilon)>0$ such that 
\[
\forall s\ge -\log \delta_0,\;\; \|(\hat w_0(s), \partial_s \hat w_0(s))- (\kappa(\hat d(0)),0)\|_{H^1\times L^2(|x|< A)}\le \epsilon.
\]
 Consider then an arbitrary $\delta$ in the interval $(0, \delta_0)$. From the finite speed of propagation, we know that $\hat u$ is defined in the truncated cone with slope one $\{(x,t),\;|\;|x|< 1+(A-1)\delta-t\mbox{ and }0\le t \le 1-\delta\}$. From the solution to the Cauchy problem in the cone, there exists $\epsilon_1(A,\delta)>0$ such that for any initial data $(u_0,u_1)$ satisfying $\|(u_0, u_1)-(\hat u_0, \hat u_1)\|_{H^1\times L^2} \le \epsilon_1$, the solution $u(x,t)$ of equation \eqref{equ} with initial data $(u_0,u_1)$ is defined on the same truncated cone and satisfies
\begin{equation}\label{cv}
\|(w_{0,1}(-\log\delta), \partial_s w_{0,1}(-\log\delta))- (\kappa(\hat d(0)),0)\|_{H^1\times L^2(|x|< A)}\le 2\epsilon.
\end{equation}
This means in particular that
\begin{equation}\label{lb}
\forall |x|\le \delta,\;\;T(x) \ge 1-\delta,
\end{equation}
Following \eqref{lb}, we note that the upper bound on $T(x)$ will follow if we prove that for $|x|$ small enough, 
\begin{equation}\label{hada}
E(w_{x,1+(A'-1) \delta}(\sigma_{A,\delta}))\le -1\mbox{ where }\sigma_{A,\delta} = - \log(A'\delta).
\end{equation}
Indeed, if this holds, then the blow-up criterion of Lemma \ref{lemcrit} applies and we see that the function $w_{x,1+(A'-1) \delta}(z,\sigma)$ cannot be defined for all $(z,\sigma) \in B(0,1)\times [\sigma_{A,\delta}, +\infty)$. From the similarity variables' definition \eqref{defw}, we see that $u(x,t)$ cannot be defined in all the cone $\q C_{x,1+(A'-1) \delta,1}$, which means that
\begin{equation}\label{ub}
T(x)\le 1+ (A'-1) \delta.
\end{equation}
Recalling the lower bound \eqref{lb}, we get the conclusion of Lemma \ref{lemcont}. Thus, it remains to prove \eqref{hada} in order to conclude.\\
Using \eqref{defw}, we see that for all $|z|<1$,
\[
 w_{x,1+(A'-1) \delta}(z,\sigma_{A,\delta})= (A'\delta)^{\frac 2{p-1}}u(x+zA'\delta, 1-\delta)=(A')^{\frac 2{p-1}}w_{0,1}(A' z+\frac x\delta, -\log \delta).
\]
Imposing that $|x|<\delta(A-A')$, we see that $|A' z+\frac x\delta|<A$.
Therefore, we can use \eqref{cv} and see that
\begin{equation}\label{prox}
\|(w_{x,1+(A'-1) \delta}(\sigma_{A,\delta}), \partial_s w_{x,1+(A'-1) \delta}(\sigma_{A,\delta}))-(w_-(\tau_{A,\delta}(x)), \partial_s w_-(\tau_{A, \delta}(x)))\|_{\q H}\le C(A) \epsilon
\end{equation}
 where $w_-(y,s)$ is a particular blow-up solution of equation \eqref{eqw} given for all $|y|<1$ and $s<\log(1-|\hat d(0)|)$ by:
\begin{equation}\label{defw-}
w_-(y,s) = \kappa_0 \frac{(1-|\hat d(0)|^2)^{\frac 1{p-1}}}{(1-e^s+\hat d(0)\cdot y)^{\frac 2{p-1}}}.
\end{equation}
and
\[
\tau_{A, \delta}(x)=\log\left(1-\frac 1{A'} - \frac{\hat d(0)\cdot x}{A' \delta}\right)
\in [\underline \tau(A), \bar \tau(A)]
\]
where
\begin{equation}\label{tbtb}
\underline \tau(A)=\log(1-\frac 1{A"})\mbox{ and }\bar \tau(A)=\log(1-\frac 1A).
\end{equation}
Since $\bar\tau(A)<\log\left(1-|\hat d(0)|\right)$, the blow-up time of $w_-$ \eqref{defw-}), there exists a constant $K^*(A)>0$ such that
\[
\left\|(w_-(\tau_{\alpha, \delta}(x)), \partial_s w_-(\tau_{\alpha, \delta}(x)))\right\|_{\q H}\le K^*(A) \equiv \sup_{\underline\tau(A)\le s\le \bar\tau(A)}\left\|(w_-(s), \partial_s w_-(s))\right\|_{\q H}.
\]
Using \eqref{prox}, item (ii) of Lemma \ref{lemcrit}, together with the monotonicity of the Lyapunov functional, we see that 
\begin{equation}\label{barriere}
E\left(w_{x,1+(A'-1) \delta}(\sigma_\delta)\right)\le
E(w_-(\tau_{A, \delta}(x)))+C^*(A)\epsilon 
\le E(w_-(\underline\tau_{A}(x)))+C^*(A)\epsilon 
\end{equation}
for some $C^*(A)>0$.
Now, we will fix the constants $A$ and $\epsilon$ so that we get \eqref{hada}:\\
- From item (iii) in Lemma \ref{lemcrit}, \eqref{tbtb} and the expression of $A"=A"(A)$, we fix $A$ close enough to $\frac 1{|\hat d(0)|}$ (large enough if $\hat d(0)=0$) so that
\[
E(w_-(\underline\tau_{A}(x)))\le -2;
\]
- Then, we fix $\epsilon(A)=\frac 1{C^*(A)}$.\\
With these two items, we get \eqref{hada}. Since \eqref{hada} implies \eqref{ub} as explained earlier, and since we have already proved the lower bound in \eqref{lb}, we conclude the proof of Lemma \ref{lemcont}.
\end{proof}

\bigskip

{\bf Step 2: Differentiability at non-characteristic points}

In this step, we show the following:
\begin{lem}[Differentiability at non-characteristic points]\label{lemdiff}
For all $\delta>0$, there exists  $\epsilon_2>0$ and $\eta_2>0$ such that if $\|(u_0, u_1)-(\hat u_0, \hat u_1)\|_{H^1\times L^2} \le \epsilon_2$, 
$|x|\le \eta_2$ and $x\in \RR$, then, the following holds:\\
(i)  $x\in \RR_0$. More precisely, \eqref{profileN} holds for $w_{x}$ with $|d(x)-\hat d(0)|\le \delta$ and $e(x)=1$.\\
(ii) The function $x\mapsto T(x)$ is differentiable at $x$ with $\nabla T(x) = d(x)$.\\
(iii) For all $\omega \in \m S^{N-1}$, there exists $\theta^*(x,\omega)>0$ such that for all $\theta \in [0, \theta^*)$, $x+\theta \omega\in \RR_0$ and
\begin{equation}\label{ident}
\forall \theta \in [0, \theta^*],\;\;|T(x+\theta \omega)-T(x) - \theta\hat d(0)\cdot \omega|\le \delta \theta.
\end{equation}
\end{lem}
\begin{nb} As for $u(x,t)$ and $T(x)$, the parameter $\theta^*(x,\omega)$ depends also on $(u_0, u_1)$. That dependence is omitted as we did for $u(x,t)$ and $T(x)$.
\end{nb}
\begin{proof} (i) This is a twin statement of Lemma \ref{lemp1}: it extends the previous statement to all initial data in some neighborhood of $(\hat u_0, \hat u_1)$. Thanks to the continuity of the blow-up time with respect to initial data, stated in Lemma \ref{lemcont}, and to the continuity (with respect to initial data) of the wave flow in light cones at a given time, the proof of Lemma \ref{lemp1} extends with no difficulty. As a conclusion, we get the existence of $\epsilon_2>0$ and $\eta_2>0$ such that if $\|(u_0, u_1)-(\hat u_0, \hat u_1)\|_{H^1\times L^2} \le \epsilon_2$, 
$|x|\le 2\eta_2$ (and not just $\eta_2$) 
and $x\in \RR$, then, item (i) holds.\\
(ii) This is a direct application of Proposition \ref{propdiff}. This item holds also for all $|x|\le 2\eta_2$.\\
(iii) Consider $\omega\in \m S^{N-1}$ and $|x|\le \eta_2$. Since $\RR_0$ is open by Theorem \ref{threg}, there exists $\bar\theta(x,\omega)>0$ such that for all $\theta \in [0, \bar \theta)$, $x+\theta \omega\in \RR_0\cap \bar B(0, 2\eta_2)$. Applying items (i) and (ii) (which both hold when $x$ is in the larger ball $\bar B(0, 2\eta_2)$), we see that $x\mapsto T(x)$ is differentiable at $x+\theta \omega$ with $|\nabla T(x+\theta \omega) - \hat d(0)|\le \delta$. Integrating this identity for $\theta \in [0, \bar\theta]$ yields \eqref{ident} with the additional property that
\begin{equation}\label{cdtn}
x+\bar \theta\omega \in \bar B(0,2\eta_2).
\end{equation}
 Omitting this condition, we define $\theta^*(x,\omega)\ge \bar \theta$ as the maximal $\theta$ such that item (iii) holds, without caring about the condition \eqref{cdtn}.
\end{proof}

\bigskip

{\bf Step 3: End of the proof of Theorem \ref{thstab}}

Consider $\delta=\frac{1-|\hat d(0)|}{10}$ and initial data $(u_0, u_1)$ such that
\[
\|(u_0, u_1)-(\hat u_0, \hat u_1)\|_{H^1\times L^2} \le \min(\epsilon_1,\epsilon_2(\delta)),
\]
 $u(x,t)$ the corresponding solution of equation \eqref{equ}, $x\mapsto T(x)$ its blow-up surface, $\SS$ the set of its characteristic points, $\RR$ the set of its non-characteristic points and $\RR_0$ the subset of $\RR$ such that the solution converges to some $\pm \kappa(d)$ is similarity variables. We would like to show that the closed ball $\bar B(0,\eta)\subset \RR_0$ where $\eta=\min(\eta_1, \eta_2(\delta))$.\\
We proceed by contradiction, and assume that there exists $|x_0|\le \eta$ such that $x_0\in \SS$. Let us consider such an $x_0$ with a minimal norm $|x_0|$. In this case, we have $B(0,|x_0|)\subset \RR$ and for all $r\ge |x_0|$, the set $\{|x|\le r,\;\;x\in \SS\}\neq \emptyset$. Therefore, we may introduce
\begin{equation}\label{defmin}
\mbox{for any }r\in[|x_0|, \eta],\;\;
\tilde T(r)=\min_{|x|\le r,\;\;x\in \SS}T(x).
\end{equation}
Let us first make the following observation on the localization of non-characteristic points with respect to minimizing points of $\tilde T$:
\begin{lem} \label{lemint} Consider $r\in [|x_0|, \eta]$ and $x_1\in \SS$ such that $|x_1|\le r$ and $\tilde T(r) = T(x_1)$. Consider also $x_2\in \RR\cap \bar B(0,\eta)$ and $\frac{9|\hat d(0)|+1}{10}<\delta_2<1$. Then, \\
- either $T(x_2)\ge T(x_1) -\delta_2|x_2-x_1|$;\\
- or there exists $x^*\in \SS$ in the open segment $(x_2, x_1)$ such that 
$T(x^*)\le T(x_1) - \delta_2 |x^*-x_1|$.\\
Moreover, if $|x_2|\le r$, then only the first case occurs.
\end{lem}
\begin{proof} Consider $r\in [|x_0|, \eta]$ and $x_1\in \SS$ such that $\tilde T(r) = T(x_1)$. Consider also $x_2\in \RR\cap \bar B(0,\eta)$ and $\frac{9|\hat d(0)|+1}{10}<\delta_2<1$. Let us assume that 
\begin{equation}\label{above}
T(x_2)< T(x_1) -\delta_2|x_2-x_1|
\end{equation}
and prove the existence of some $x^*\in \SS$ in the open segment $(x_2, x_1)$ such that $T(x^*)\le T(x_1)-\delta_2|x_1-x^*|$.\\
Since $|x_2|\le\eta \le \eta_2$,
 we see from item (i) in Lemma \ref{lemdiff} that $x_2 \in \RR_0$.
Since $x_1\neq x_2$ from \eqref{above}, introducing $\omega =\frac{x_1-x_2}{|x_1-x_2|}\in \m S^{N-1}$, we see from item (iii) in Lemma \ref{lemdiff} that for some maximal $\theta^*(x_2,\omega)>0$, we have for all $\theta \in [0, \theta^*)$, $x_2+\theta \omega\in \RR_0$, and
\begin{equation}\label{proche}
\forall \theta \in [0, \theta^*],\;\;
|T(x_2+\theta \omega)-T(x_2) - \theta\hat d(0)\cdot \omega|\le \delta \theta\mbox{ with }\delta = \frac{1-|d(0)|}{10}.
\end{equation}
Since $x_1\in \SS$, we clearly have $0<\theta^*(x_2, \omega)\le |x_1-x_2|$.\\
Assume by contradiction that $\theta^*=|x_1-x_2|$. Then, $x_1=x_2+\theta^*\omega$. Applying \eqref{proche} with $\theta=\theta^*$, we see that $|T(x_1)-T(x_2)|\le \frac{1+9|\hat d(0)|}{10}|x_1-x_2|$. Since $x_1\neq x_2$ and $\frac{9|\hat d(0)|+1}{10}<\delta_2$,  a contradiction follows from \eqref{above}.\\
Therefore, $0<\theta^*<|x_1-x_2|$,  hence $x^*\not\in \{x_1,x_2\}$, where $x^*= x_2+\theta^*\omega$. Since $|x^*|\le \max (|x_1|, |x_2|)\le \eta \le \eta_2$, and $\theta^*$ is maximal,
it follows from item (iii) in Lemma \ref{lemdiff} that $x^*\in \SS$.
Since $|x_1-x_2|=|x_1-x^*|+|x^*-x_2|$, using \eqref{proche} and \eqref{above}, we see that
\begin{align}
T(x^*)&\le T(x_1) +\frac{1+9|\hat d(0)|}{10}|x^*-x_2| - \delta_2(|x_1-x^*|+|x^*-x_2|)\nonumber\\
&\le T(x_1)-(\delta_2 -\frac{1+9|\hat d(0)|}{10}) |x^*-x_2|-\delta_2|x_1-x^*|\nonumber\\
&\le T(x_1)-\delta_2|x_1-x^*|<T(x_1).\label{second}
\end{align}
By definition of the minimum in \eqref{defmin}, we must have $|x^*|>r$. Since $x^*$ is on the segment $[x_2, x_1]$ and $|x_1|=r$, this implies that $|x_2|>r$. In particular, this means that if we have assumed that $|x_2|\le r$ at the beginning of the proof, then estimate \eqref{second} cannot occur, thus, only the first case occurs in Lemma \ref{lemint}. 
This concludes the proof of Lemma \ref{lemint}.
\end{proof}
From the definition of $\tilde T(r)$ given in \eqref{defmin}, two cases then arise:

\bigskip

\noindent{\bf Case 1}: There exists $r\in [|x_0|, \eta]$ such that the minimum is achieved in the open ball $B(0,r)$, say, at some $x_1\in \SS$ such that
\begin{equation}\label{minx1}
|x_1|<r\mbox{ and }T(x_1) = \min_{|x|\le r,\;\;x\in \SS}T(x).
\end{equation}
 Clearly, there exists $x_2\in B(x_1,r-|x_1|)$ such that
\begin{equation}\label{above1}
T(x_2)< T(x_1) -\frac{3+|\hat d(0)|}4|x_2-x_1|,
\end{equation}
otherwise $x_1\in \RR$. Since $|x_2|<|x_1|+r-|x_1|=r\le \eta$ and $T(x_2)<T(x_1)$, we see from \eqref{minx1} that necessarily $x_2\in \RR$ and $x_2\neq x_1$. Noting that $|x_2|\le r$, we see from Lemma \ref{lemint} applied with $\delta_2 = \frac{3+|\hat d(0)|}4$ that \eqref{above1} cannot occur, and a contradiction follows. 

\bigskip

\noindent {\bf Case 2}: For all $r\in [|x_0|, \eta]$, the minimum in \eqref{defmin} is not achieved in the open ball. This means that  
\[
\forall r\in [|x_0|, \eta],\;\; \tilde T(r)=\min_{|x|\le r,\;\;x\in \SS}\tilde T(x)=\min_{|x|= r, \;\;x\in \SS}\tilde T(x).
\]
Note that by construction, $r\mapsto\tilde T(r)$ is a nonincreasing function. Moreover, since $T$ is $1$-Lipschitz, we easily see that the same holds for $\tilde T$. 
We claim that 
\[
\forall r\in [|x_0|,\eta],\;\;\tilde T(r) = \tilde T(\eta)-r+\eta.
\] 
Indeed, the proof follows from our argument in Lemma 4.2 page 614 in \cite{MZajm11}. Let us recall it in the following. Assume by contradiction that for some $m'\neq 1$ and $r'\in [|x_0|, \eta)$, we have 
\begin{equation}\label{contradiction}
\tilde T(r') = \tilde T(\eta)-m'(r'-\eta).
\end{equation}
 Since $\tilde T$ is $1$-Lipschitz and nonincreasing, it follows that $0\le m'<1$.
Considering a family of lines of slope $-\frac{1+m'}2$ growing from below, we find the highest line that stays under the graph of $\tilde T$ on the interval $[r',\eta]$. In other words, there is $r_1\in [r',\eta]$ such that 
\begin{equation}\label{plusbas}
\forall r\in [r',\eta],\;\;\tilde T(r) \ge -\frac{1+m'}2(r-r_1)+\tilde T(r_1).
\end{equation}
If $r_1=\eta$, then applying this inequality with $r=r'$ and recalling that $r'<\eta$, we see that a contradiction follows from \eqref{contradiction}.\\
If $r_1\in [r', \eta)$, we consider $x_1\in \SS$ such that $|x_1|=r_1$ and $\tilde T(r_1) = T(x_1)$. Note that $|x_1|<\eta$. The following claim allows us to conclude:
\begin{cl} \label{cltech}Consider an arbitrary $x_2\in \bar B(x_1, \eta-r_1)$. Then, we have the following cases:\\
(i) if $x_2\in \SS$ and $|x_2|\le r_1$, then $T(x_2) \ge T(x_1)$;\\
(ii) if $x_2\in \SS$ and $|x_2|>r_1$, then $T(x_2) \ge T(x_1) -\frac{1+m'}2|x_2-x_1|$;\\
(iii) if  $x_2 \in \RR$ and $|x_2|\le r_1$, then $T(x_2) \ge T(x_1) - \frac{8|\hat d(0)|+1}9|x_2-x_1|$;\\
(iv)  if  $x_2 \in \RR$ and $|x_2|> r_1$, then either $T(x_2) \ge T(x_1) - \delta_1|x_2-x_1|$ or there exists $x^*(x_2)\in \SS$ in the open segment $(x_2,x_1)$ such that $T(x^*) \le T(x_1) - \delta_1|x^*-x_1|$,  where $\delta_1 = \max\left(\frac{3+m'}4, \frac{8|\hat d(0)|+1}9\right)$.
\end{cl}
Indeed, if for all $x_2\in \bar B(x_1, \eta-r_1)$ with $x_2\in \RR$ and $|x_2|>r_1$, the first case of item (iv) occurs, then, we see that 
\begin{equation*}
\mbox{if }|x_2-x_1|\le \eta-r_1,\mbox{ then }T(x_2) \ge T(x_1) - \delta_1|x_2-x_1|.
\end{equation*}
Since $m'<1$ and $|\hat d(0)|<1$, we also have $\delta_1<1$, and we see that $x_1\in \RR$, which is a contradiction.\\ 
Now, if for some $x_2\in \bar B(x_1, \eta-r_1)$ with $x_2\in \RR$ and $|x_2|>r_1$, the second case of item (iv) occurs, then, 
 we see that for some $x^*(x_2)$ in the open segment $(x_2,x_1)$, we have $T(x^*) \le T(x_1) - \delta_1|x^*-x_1|$, $|x^*-x_1|\le |x_1-x_2|\le \eta-r_1$, $x^*\in \SS$ and $|x^*|>\min(|x_2|,|x_1|)=r_1$. Applying item (ii) to $x^*$, we get a contradiction. Let us then prove Claim \ref{cltech} in order to conclude the proof of Theorem \ref{thstab}.
\begin{proof}[Proof of Claim \ref{cltech}]
Take $x_2\in \bar B(x_1, \eta-r_1)$. Note that $|x_2|\le |x_2-x_1|+|x_1|\le \eta-r_1+r_1=\eta$.\\
(i) If $x_2\in \SS$ and $|x_2|\le r_1$, then we have from the definition \eqref{defmin} of $\tilde T$ that $T(x_2) \ge \tilde T(r_1) = T(x_1)$. \\
(ii) If $x_2\in \SS$ and $|x_2|>r_1$, then we use again \eqref{defmin} and \eqref{plusbas} to write
\[
T(x_2) \ge \tilde T(|x_2|) \ge \tilde T(r_1)-\frac{1+m'}2(|x_2|-r_1)
\ge T(x_1) -\frac{1+m'}2|x_2-x_1|.
\] 
(iii) If $x_2 \in \RR$ and $|x_2|\le r_1$, we take $\delta_2 = \frac{8|\hat d(0)|+1}9>\frac{9|\hat d(0)|+1}{10}$. From Lemma \ref{lemint}, we see that only the first case in that lemma holds, which yields item (iii).\\
(iv) If $x_2\in \RR$ and $|x_2|>r_1$, applying Lemma \ref{lemint} with $\delta_2 = \delta_1$, we get the conclusion. This concludes the proof of Claim \ref{cltech}. 
\end{proof}
Since Claim \ref{cltech} concludes the proof of Theorem \ref{thstab}, this concludes the proof of Theorem \ref{thstab} too, assuming that Theorem \ref{thrig} holds.
\end{proof}

\subsection{Stability of the existence of a minimal blow-up time in $\RR_0$}
We prove Corollary \ref{stabmin} here, assuming that Theorem \ref{thrig} holds. Since Theorems \ref{threg} and \ref{thstab} together with Sections \ref{secreg} and Section \ref{secthstab} hold whenever Theorem \ref{thrig} holds, we may use them in our proof.

\medskip

\begin{proof}[Proof of Corollary \ref{stabmin} assuming that Theorem \ref{thrig} holds] 
Consider $\hat u(x,t)$ a blow-up solution of equation \eqref{equ} with initial data $(\hat u_0, \hat u_1)\in H^1\times L^2$, and $\hat x_0\in \hat \RR_0$. 
Applying Theorem \ref{thstab}, we see that for some $\hat \epsilon_0>0$ and for any $(u_0, u_1)$ such that $\|(u_0,u_1)-(\hat u_0, \hat u_1)\|_{H^1\times L^2}\le \hat \epsilon_0$,
the solution $u(x,t)$ with initial data $(u_0, u_1)$ blows up in finite time with $B(\hat x_0, \hat \epsilon_0)\subset \RR_0$. 
Assuming in addition that the function $x\mapsto \hat T(x)$ achieves a strict local minimum in $\hat x_0$ (the case of a strict local maximum follows exactly in the same way), we get the existence of some $0<\epsilon_1<\hat \epsilon_0$ and $\delta_1>0$ such that 
\[
\min_{|x-\hat x_0|=\epsilon_1}\hat T(x) \ge \hat T(\hat x_0)+3\delta_1.
\]
Consider $x\in \m R^N$ such that $|x-\hat x_0|=\epsilon_1$. Since $x$ and $\hat x_0$ are both in $\hat \RR_0$, we may apply the continuity of the blow-up time result stated in Lemma \ref{lemcont} and 
get the existence of some $\eta_x>0$ 
and some $\hat \epsilon_x\le \hat \epsilon_0$ such that if $\|(u_0,u_1)-(\hat u_0, \hat u_1)\|_{H^1\times L^2}\le \hat \epsilon_x$, then
\[
\min_{|y-x|\le \eta_x}T(y) \ge \hat T(\hat x_0) + 2\delta_1\mbox{ and }T(\hat x_0)\le \hat T(\hat x_0) + \delta_1.
\]
Since the sphere $S(\hat x_0, \epsilon_1)$ is compact, there is a finite number of points $x_1,\dots, x_k$ of points in that sphere, for some $\k\in \m N^*$,  such that  
$S(\hat x_0, \epsilon_1)\subset \cup_{i=1,\dots,k}\hat B(x_i,\eta_{x_i})$. 
Introducing $\epsilon_0' = \min\{\hat \epsilon_0, \epsilon_{x_i}\;|\;i=1,\dots,k\}$, we see that for all $(u_0, u_1)$ such that $\|(u_0,u_1)-(\hat u_0, \hat u_1)\|_{H^1\times L^2}\le \hat \epsilon_0'$, 
\begin{equation}\label{separation}
\min_{|x-\hat x_0|=\epsilon_1}T(y) \ge \hat T(\hat x_0) + 2\delta_1\mbox{ and }T(\hat x_0)\le \hat T(\hat x_0) + \delta_1.
\end{equation}
Since $B(\hat x_0, \hat \epsilon_0)\subset \RR_0$, using Theorem \ref{threg}, we see that  $x\mapsto T(x)$ is $C^1$ in $B(\hat x_0, \hat \epsilon_0)$. Since $\bar B(\hat x_0, \epsilon_1)\subset B(\hat x_0, \hat \epsilon_0)$, we see from \eqref{separation} that the function $x\mapsto T(x)$ must achieve a local minimum $x_0$ in the open ball $B(\hat x_0, \epsilon_1)$ with $\nabla T(x_0)=0$. Since $x_0\in \RR_0$, applying again Theorem \ref{threg}, we see that $w_{x_0}(s) \to \pm \kappa_0$ as $s\to \infty$. This concludes the proof of Corollary \ref{stabmin}.  
\end{proof}

\section{A rigidity theorem for equation \eqref{eqw}}\label{secrig}
This section is devoted to the proof of Theorems \ref{thrig} and \ref{thrig}'. Note that the proof uses the dynamical system formulation and the modulation technique given in \cite{MZtds} and recalled in Section \ref{secdyn} above. Let us first derive Theorem \ref{thrig} from Theorem \ref{thrig}', then prove Theorem \ref{thrig}'.

\medskip
 
\begin{proof}[Proof of Theorem \ref{thrig} assuming Theorem \ref{thrig}']
Assume that Theorem \ref{thrig}' holds and consider $w(y,s)$ a solution of equation \eqref{eqw} defined for all $|y|<A^*$ and $s\in \m R$ for some $A^*>1$ and satisfying \eqref{hypw}.
Introducing 
\[
u(x,t) = (-t)^{-\frac 2{p-1}}w(y,s)\mbox{ where }y=\frac x{-t}\mbox{ and }s=-\log(-t),
\]
we see that by definition, $w$ is the similarity version of $u$ at $(0,0)$ (in other words $w_{0,0}=w$), and that $u(x,t)$ satisfies the hypotheses of Theorem \ref{thrig}' with $x^*=0$, $T^*=0$ and $\delta^*=\frac 1{A^*}$. Thus, the conclusion of Theorem \ref{thrig} follows from the conclusion of Theorem \ref{thrig}'.
\end{proof}

We now give the proof of Theorem \ref{thrig}'.

\medskip

\begin{proof}[Proof of Theorem \ref{thrig}']
Consider $u(x,t)$ a solution of equation \eqref{equ} defined for all $(x,t) \in \q C_{x^*, T^*, \delta^*}$ \eqref{defcone}. From the symmetries of equation \eqref{equ}, we may assume that $x^*=0$, $T^*=0$ and $e^*=1$. Introducing $W=w_{0,0}$ and $A^*=\frac 1{\delta^*}>1$, we assume that for all $s\in \m R$:
\begin{equation}\label{hypo}
\|(W(s), \partial_s W(s))\|_{H^1\times L^2 (|y|<A^*)}\le M^*\mbox{ and }
\forall |y|<1,\;\;W(y,s) = \kappa(d^*,y),
\end{equation}
for some $|d^*|<1$ and $M^*>0$. We would like to prove that $u(x,t)$ is explicitly given by 
\begin{equation}\label{result}
u(x,t) = \kappa_0 \frac{(1-|d^*|^2)^{\frac 1{p-1}}}{(-t +d^*\cdot x)^{\frac 2{p-1}}}.
\end{equation}
Let us rapidly present the argument of the proof. The details will be given later. We proceed in two parts:\\
- In Part 1, we consider the similarity variables' transformation \eqref{defw} $w_{a,d^*\cdot a}$ (or $w_a$ for simplicity) around $(a,d^*\cdot a)$ where $a\in \m R^N$. Note that this point is on the singular line of the right-hand side of the targeted result \eqref{result}.
Since
\begin{equation}\label{waw}
w_a(y,s)=\lambda^{-\frac 2{p-1}} W\left(\frac{y+ae^s}\lambda,s - \log \lambda\right) \mbox{ with }\lambda =1-d^*\cdot a e^s
\end{equation}
from \eqref{defw}, we translate the information \eqref{hypo} into estimates for $w_a$. In particular, we show that
\begin{equation}\label{poss0}
\|(w_a(s), \partial_s w_a(s))-(\kappa(d^*),0)\|_{\H} =O(e^{\frac{\alpha s}2})\mbox{ as }s\to -\infty.
\end{equation}
- In Part 2, using \eqref{poss0}, we see that Proposition \ref{propexpo} applies, and we get the exponential decay property for $q(y,s)$, where 
\begin{equation}\label{defqq}
q(y,s) = \vc{w_a(y,s)}{\partial_s w_a(y,s)}-\kappa^*(d(s),\nu(s),y)
\end{equation}
and the parameters $|d(s)|<1$ and $\nu(s) >-1+|d(s)|$ are of class $C^1$. 
Since $s\to -\infty$ and $q$ is bounded, we must have
\begin{equation}\label{q0}
q(y,s)\equiv 0\mbox{ for all }|y|<1,
\end{equation}
on the one hand. On the other hand, from the fact that $W(y,s)\equiv \kappa(d^*,y)$ for all $|y|<1$ (see \eqref{hypo}) and \eqref{waw}, we see that
\begin{equation}\label{wak}
w_a(y,s) \equiv \kappa(d^*,y)\mbox{ on some non empty open set }K_{a,s}\subset B(0,1)
\end{equation}
 defined below in \eqref{defkas}. Therefore, it follows from \eqref{q0} and \eqref{defqq} that
\[
d(s)\equiv d^*\mbox{ and }\nu(s)\equiv 0,
\]
hence $w_a(y,s) \equiv \kappa(d^*,s)$ for all $|y|<1$. Since $a$ was arbitrary, we use the similarity variables' transformation \eqref{defw} to recover \eqref{result} for all $(x,t)\in \q C_{0, 0, \delta^*}$.

\medskip

{\bf Part 1: Translating the information for $w_a$}

We first claim the following:
\begin{prop}\label{prop1}For all $a\in \m R^N$, there exists  $s_1(a)\in \m R$ such that the following holds for any $s\le s_1(a)$ :\\
(i) $\|(w_a(s), \partial_s w_a(s))\|_{\H}\le C M^*$;\\
  (ii) Estimate \eqref{wak} holds and 
\[
\|(w_a(s), \partial_s w_a(s))-(\kappa(d^*),0)\|_{\H} =O(e^{\frac{\alpha s}2})\mbox{ as }s\to -\infty.
\] 
\end{prop}
\begin{proof}Consider $a\in \m R^N$.\\
(i) Taking 
\begin{equation}\label{conds}
s\le \sigma_a\equiv\min\left(-\log(2|d^*|\cdot|a|), \log  \frac{A^*-1}{|a|(1-A^*|d^*|)}\right),
\end{equation}
 we see in \eqref{waw} that $\frac 12 \le \lambda \le 2$ and $\frac{|y+ae^s|}\lambda <A^*$ whenever $|y|<1$. Since $\rho(y)\le 1$ from \eqref{defro}, this yields 
\[
\int_{|y|<1} (w_a(y,s))^2 \rho(y) dy \le 2^{\frac 2{p-1}}\int_{|z|<A^*}(W(z,s-\log \lambda))^2 dz.
\]
Since \eqref{waw} gives
\begin{align*}
\nabla w_a(y,s)& = \lambda^{-\frac 2{p-1}} \nabla W(z,s-\log \lambda)\mbox{ where }z=\frac{y+ae^s}\lambda \\
\partial_s w_a(y,s)& = \lambda^{-\frac 2{p-1}}\left((1-\frac{\lambda'}{\lambda})
\partial_s W+ae^s(1-\frac{\lambda'}{\lambda^2})\cdot \nabla W-\frac 2{p-1}\frac{\lambda'}{\lambda}W\right)(z,s-\log \lambda),
\end{align*}
and $|\lambda'|\le |d^*||a|e^s \le \frac 12$ from \eqref{conds}, we similarly get 
\[
\|(w_a(s),\partial_s w_a(s))\|_{\H} \le C\|(w_a(s),\partial_s w_a(s))\|_{H^1\times L^2(|y|<A^*)}.
\]
Using the bound in \eqref{hypo}, we conclude the proof of (i).\\
(ii) The key idea here is the fact that $\kappa(d^*)$ is invariant under the transformation \eqref{waw}.
Introducing the following intersection between balls
\begin{equation}\label{defkas}
K_{a,s}=\{|y|<1\}\cap \{|z|<1\}\mbox{ where }z=\frac{y+ae^s}\lambda\mbox{ and }\lambda = 1-d^*\cdot a e^s,
\end{equation}
we see from \eqref{conds} that $K_{a,s}$ is a non empty open set for $|s|$ large enough. Moreover, using \eqref{waw}, \eqref{hypo} and the definition \eqref{defkd} of $\kappa(d,y)$, we see that for all $y\in K_{a,s}$,
\begin{align}
 w_a(y,s)=&\lambda^{-\frac 2{p-1}}W(z,s-\log \lambda)=\lambda^{-\frac 2{p-1}}\kappa(d^*,z)\label{proofwak}\\
=&\frac{\kappa_0(1-|d^*|^2)^{\frac 1{p-1}}}{\lambda^{\frac 2{p-1}}\left(1+d^* \cdot \frac{(y+ae^s)}\lambda\right)^{\frac 2{p-1}}}
=\frac{\kappa_0(1-|d^*|^2)^{\frac 1{p-1}}}{\left(\lambda+d^* \cdot y+d^*\cdot ae^s\right)^{\frac 2{p-1}}}=
\kappa(d^*,y)\nonumber
\end{align}
and \eqref{wak} holds.\\
Therefore, from the bound in \eqref{hypo}, we write
\begin{align}
&\|(w_a(s), \partial_s w_a(s))-(\kappa(d^*,y),0)\|_{\H}=
\|(w_a(s), \partial_s w_a(s))-(\kappa(d^*,y),0)\|_{\H(y\not\in K_{a,s})}\nonumber\\
\le& \max_{|y|<1,\;y\not\in K_{a,s}}\sqrt{\rho(y)}(\|(w_a(s), \partial_s w_a(s))\|_{H^1\times L^2(|y|<1)}+\|\kappa(d^*,y)\|_{H^1(|y|<1)})\nonumber\\
\le& (CM^*+C(d^*)) \max_{|y|<1,\;y\not\in K_{a,s}}\sqrt{\rho(y)}.\label{warda}
\end{align}
Since when $|y|<1$ and $y\not\in K_{a,s}$, we have $|y+ae^s|\ge \lambda = 1- d^*\cdot a e^s$, it follows that $0\le 1-|y|\le |a|(1+|d^*|)e^s$, hence, by definition \eqref{defro} of $\rho$, it follows that
\[
\max_{|y|<1,\;y\not\in K_{a,s}}\sqrt{\rho(y)}\le (|a|(1+|d^*|)e^s)^{\frac \alpha 2}
\]
and (ii) follows from \eqref{warda}. This concludes the proof of Proposition \ref{prop1}.  
 \end{proof}

\bigskip
{\bf Part 2: A modulation technique as $s\to -\infty$}

From item (ii) in Proposition \ref{prop1}, we see that we can apply Proposition \ref{propexpo} and obtain for some $\bar \sigma(a)\in \m R$,
\begin{equation}\label{qsmall}
\forall s'\le s\le \bar \sigma(a),\;\;
\|q(s)\|_{\H}\le Ce^{-\mu_0(s-s')},
\end{equation}
where $q$ is defined by \eqref{defqq}, for some $C^1$ parameters $|d(s)|<1$ and $\nu(s) >-1+|d(s)|$ (please use the remark following Proposition \ref{propexpo} to see that the parameters $d(s)$ and $\nu(s)$ do not depend on $s'$ at all).\\
Taking $s\le \bar \sigma(a)$ and letting $s'\to \infty$, we see from \eqref{defqq} that 
\begin{equation}\label{q00}
q(y,s) =0\mbox{ hence }w_a(y,s)=\kappa^*_1(d(s),\nu(s),y)\mbox{ for all }|y|<1,
\end{equation}
on the one hand. 
On the other hand, we recall that we have already proved in item (ii) of Proposition \ref{prop1} that \eqref{wak} holds, namely that
\begin{equation}\label{wak2}
w_a(y,s) \equiv \kappa(d^*,y)\mbox{ for all }y\in K_{a,s}, 
\end{equation}
a non empty open set in $B(0,1)$.
 and
 defined in \eqref{defkas}. Therefore, by definitions \eqref{defkd} and \eqref{defk*} of $\kappa(d,y)$ and $\kappa^*_1(d,\nu,y)$, we see that 
\begin{equation}\label{ident0}
\forall y\in K_{a,s},\;\;\kappa^*_1(d(s),\nu(s),y)=\kappa(d^*,y)=\kappa^*_1(d^*,0,y).
\end{equation}
Since $K_{a,s}$ is a non empty open set, we see from \eqref{defk*} that we can identity the parameters and get
\[
d(s)=d^*\mbox{ and }\nu(s)=0,
\]
hence \eqref{ident0} extends to all $|y|<1$, namely that 
\begin{equation*}
w_a(y,s) \equiv \kappa(d^*,y)\mbox{ for all } |y|<1.
\end{equation*}
Recalling that $w_a$ is the short form of $w_{a,d^*\cdot a}$, we use the definition of similarity variables \eqref{defw} to recover
\begin{align}
u(x,t)&= (d^*\cdot a-t)^{\frac 2{p-1}}w_{a, d^*\cdot a}\left(\frac{x-a}{d^*\cdot a-t}, -\log(d^*\cdot a-t)\right)\nonumber\\
&=(d^*\cdot a-t)^{\frac 2{p-1}}\kappa\left(d^*,\frac{x-a}{d^*\cdot a-t}\right)
=(d^*\cdot a-t)^{\frac 2{p-1}}\kappa_0\frac{(1-|d^*|^2)^{\frac 1{p-1}}}{(1+d^*\cdot \frac{x-a}{d^*\cdot a-t})^{\frac 2{p-1}}}\nonumber\\
&=\kappa_0\frac{(1-|d^*|^2)^{\frac 1{p-1}}}{(d^*\cdot a-t +d^*\cdot(x-a))^{\frac 2{p-1}}}
=\kappa_0\frac{(1-|d^*|^2)^{\frac 1{p-1}}}{(-t +d^*x)^{\frac 2{p-1}}}\label{identu}
\end{align}
for all $(x,t)\in \q C_{a, d^*\cdot a, 1}\cap \q C_{0,0,\delta^*}$ and $t$ less than some $t_0(a)\in \m R$. From the uniqueness of the solution of equation \eqref{equ} in light cones, the identity \eqref{identu} extends to all $(x,t)\in \q C_{a, d^*\cdot a, 1}\cap \q C_{0,0,\delta^*}$. Since $a$ was arbitrary and 
\[
\displaystyle\cup_{a\in \m R}\q C_{a, d^*\cdot a, 1}\cap \q C_{0,0,\delta^*}=\q C_{0,0,\delta^*},
\]
we obtain the desired estimate \eqref{result}. This concludes the proof of Theorem \ref{thrig}'.

\end{proof}

\appendix

\section{Proof of Lemma \ref{lemliouv}}\label{applemliouv}
This section is devoted to the proof of Lemma \ref{lemliouv}. The case $N=1$ is treated in Claim 2.3 page 62 in \cite{MZcmp08}. Thus, we assume that
\[
N\ge 2
\]
in this section.\\
As in one space dimension, the result follows from the weak continuity of the solutions of equations \eqref{equ} in $H^1\times L^2$ (here and in the following, the domain on which we consider Sobolev spaces is $\m R^N$ unless otherwise specified). More precisely, consider a sequence of solutions $u_n$ to equation \eqref{equ} defined in 
\begin{equation}\label{defct}
\C_{t_0} =  \{(\xixi, \tautau)\;\mid\;0<\tautau< t_0\mbox{ and
}|\xixi|<A(1-\tautau)\}\mbox{ where }\tautau_0<1,
\end{equation}
such that for all $\tautau \in [0, \tautau_0]$,
 \begin{align}
&\|u_n(\tautau)\|_{L^2 (|\xixi|<A(1-\tautau))}+(1-\tautau)\|(\nabla u_n(\tautau), \partial_\tautau u_n(\tautau))\|_{L^2\times L^2 (|\xixi|<A(1-\tautau))}\le M_0(1-\tautau)^{-\frac 2{p-1}+\frac N2},\label{boundun}\\
&(u_n(0), \partial_\tautau u_n(0))\wto (z^*, z^*_1)\mbox{ as }n\to \infty,\mbox{ in }H^1\times L^2 (|x|<A)\label{limun}
\end{align}
where $M_0>0$ for some $(z^*, z^*_1)\in H^1\times L^2 (|x|<A)$. Then, we have the following:
\begin{lem}[Weak continuity with respect to initial data in $H^1\times L^2$ in some cone, for solutions to \eqref{equ}]\label{lemcont2}
There exists a solution $u(\xixi,\tautau)$ of \eqref{equ} with initial data $(z^*,z^*_1)$ defined in $\C_{\tautau_0}$ such that:\\
(a)  For all $\tautau\in [0, \tautau_0]$, $(v_n(\tautau), \partial_\tautau
v_n(\tautau)) \wto 0$ weakly in $H^1\times L^2(|\xixi|<A(1-\tautau))$, where $v_n = u_n -u$.\\
(b) $\sup_{\tautau \in [0, \tautau_0]}\|v_n(\tautau)\|_{L^2(|\xixi|<A(1-\tautau))}\to 0$ as $n\to \infty$.\\
(c) There exists $n_0\in {\m N}$ such that for all $n\ge n_0$ and $\tautau \in [0, \tautau_0]$,\\
 $\|\nabla v_n(\tautau)\|_{L^2(|\xixi|<A(1-\tautau))}+\|\partial_\tautau v_n(\tautau)\|_{L^2(|\xixi|<A(1-\tautau))}
\le 20M_0$.
\end{lem}
The derivation of Lemma \ref{lemliouv} from this result is omitted, since it follows exactly as in the one-dimensional case (see Appendix A page 78 in \cite{MZcmp08}). On the contrary, the proof of Lemma \ref{lemcont2} is different, as it uses the properties of the wave operator, which heavily depend on the dimension. Thus, we only prove Lemma \ref{lemcont2} in the following. We proceed in 3 steps:\\
- In Step 1, we give several interpolation and Strichartz estimates.\\
- In Step 2, we give a weak continuity result for solutions defined in the whole space.\\
- In Step 3, using the finite speed of propagation and a localization technique, we prove Lemma \ref{lemcont2}.

\bigskip

{\bf Step 1: Interpolation and Strichartz estimates}

The proof of Lemma \ref{lemcont2} needs various classical interpolation and Strichartz estimates for equation \eqref{equ}, which we recall in the following.

\medskip

{\it - Strichartz estimates:} Introducing the following linear wave equation with a given source:
\begin{equation}\label{lcp}
\partial_t^2 v -\Delta v = h, 
\end{equation}
we recall the following Strichartz estimates from Lemma 2.1 page 150 in Kenig and Merle \cite{KMam08}:
\begin{lem}[Strichartz estimates for equation \eqref{lcp}]\label{lemstri} There is a constant $C>0$ such that for all $\ttt>0$, we have
\[
\sup_{t\in [0, \ttt]}\left(\|v(t)\|_{\dot H^1}+\|\partial_t v(t)\|_{\dot H^1}\right)
 \le C\left(\|v(0)\|_{\dot H^1}+\|\partial_t v(0)\|_{L^2}
+\|D^{\frac 12}h\|_{L^{\frac{2(N+1)}{(N+3)}}(S_\ttt)}
\right),
\]
where $S_\ttt=\m R^N \times [0, \ttt]$.
\end{lem}

\medskip

{\it - Derivatives of differences:}  We need the following lemma by Killip and Vi\c san \cite{KVpams11}:
\begin{lem}[Derivatives of differences]\label{lemdd}Consider $F(U)=|U|^{p-1}U$, $1<q,q_1, q_2<\infty$ with $\frac 1q=\frac {p-1}{q_1}+\frac 1{q_2}$. Then, for any functions $U$ and $V$ such that the right-hand side is finite, we have 
\[
\|D^{\frac 12}\left[F(U) - F(V)\right]\|_{L^q}\le 
\|U\|_{L^{q_1}}^{p-1}\|D^{\frac 12}(U-V)\|_{L^{q_2}}+
\|U-V\|_{L^{q_1}}^{p-1}\|D^{\frac 12} V\|_{L^{q_2}}.
\]
\end{lem}
\begin{proof} See Lemma 2.3 page 1809 in \cite{KVpams11} where the statement is given, and where a proof inspired by Taylor \cite{Tams00} is sketched.
\end{proof} 

\medskip

{\it - A fractional Gagliardo-Nirenberg inequality:} We will use the following fractional Gagliardo-Nirenberg inequality from Oru \cite{Ophd98}:
\begin{lem}[A fractional Gagliardo-Nirenberg inequality]\label{lemgn} Consider $0\le s_1<s_2<\infty$, $1<p_1<\infty$, $1<p_2<\infty$, and $s$ and $p$ defined by 
\begin{equation}\label{deftheta}
s=\theta s_1 +(1-\theta)s_2\mbox{ and }\frac 1p=\frac \theta{p_1}+\frac{1-\theta}{p_2}
\mbox{ for some }\theta\in[0,1].
\end{equation}
Then, for any $f\in W^{s_1,p_1}\cap W^{s_2,p_2}$, it holds that $f\in W^{s,p}$ and 
\[
\|f\|_{W^{s,p}}\le C \|f\|_{W^{s_1,p_1}}^\theta\|f\|_{W^{s_2,p_2}}^{1-\theta}.
\]
\end{lem}
\begin{proof}
This statement is a consequence of a more general statement proved by F. Oru in \cite{Ophd98}. Br\'ezis and Mironescu give the result, its consequence and proof in \cite{BMjee01} (see Lemma 3.1 and Corollary 3.2 in page 393 of that paper). 
\end{proof}

\bigskip

{\bf Step 2: Weak continuity in $H^1\times L^2$ in the whole space}

This is the aim of this step:
\begin{lem}[Weak continuity of the flow of equation \eqref{equ} in $H^1\times L^2$]\label{propg} Consider $z$ and $z_n$ solutions of equation \eqref{equ} such that 
$\supp(z_n)\cup \supp(z)\subset [0, \ttt]\times B(0, \bar A)$ and
\begin{equation}\label{boundzzn}
\forall t\in [0,\ttt],\;\;\|(z,\partial_t z)\|_{H^1\times L^2}+
\|(z_n,\partial_t z_n)\|_{H^1\times L^2}\le \bar M
\end{equation}
for some $\ttt>0$, $\bar A>0$ 
and $\bar M>0$. Assume that 
\begin{equation}\label{wconv}
z_n(0) \wto z(0)\mbox{ in }H^1\mbox{ and }\partial_t z_n(0) \wto \partial_t z(0)\mbox{ in }L^2
\end{equation}
as $n\to \infty$. Then, up to extracting a subsequence still denoted by $z_n$:\\
(i) We have
\begin{align*}
\|z_n\|_{L^{\frac{2(N+1)}{(N-2\gamma)}}(S_\ttt)}+
\|z\|_{L^{\frac{2(N+1)}{(N-2\gamma)}}(S_\ttt)}&\le C(\bar A, \ttt, \bar M),\\
\|D^{\frac 12}z_n\|_{L^{\frac{2(N+1)}{(N-1)}}(S_\ttt)}
+\|D^{\frac 12}z\|_{L^{\frac{2(N+1)}{(N-1)}}(S_\ttt)}
&\le C(\bar M),
\end{align*}
where $S_\ttt=\m R^N\times[0,\ttt]$.\\
(ii) For any $\gamma<1$ close enough to $1$,
\begin{equation}\label{conv1-}
\sup_{t\in [0,\ttt]}\|(z_n(t), \partial_t z_n(t))- (z(t), \partial_t z(t))\|_{H^{\gamma}\times H^{\gamma-1}}\to 0\mbox{ as }n\to \infty.
\end{equation}
(iii) We also have for any $\gamma<1$ close enough to $1$,
\begin{equation}\label{p+1}
\|z_n-z\|_{L^{\frac{2(N+1)}{(N-2\gamma)}}(S_\ttt)}+
\||D|^{\gamma-\frac 12}(z_n-z)\|_{L^{\frac{2(N+1)}{(N-1)}}(S_\ttt)}
\to 0\mbox{ as }n\to \infty.
\end{equation}
(iv) There exists $n_0(\bar A, \ttt, \bar M)\in \m N$ large enough, such that for all $n\ge n_0$, 
\[
\sup_{t\in [0,\ttt]}\|(z_n(t), \partial_t z_n(t))- (z(t), \partial_t z(t))\|_{\dot H^1\times L^2}\le 5 \bar M.
\]
(v) For all $t\in [0,\ttt]$,
\begin{equation}\label{convt}
z_n(t) \wto z(t)\mbox{ in }H^1\mbox{ and }\partial_t z_n(t) \wto \partial_t z(t)\mbox{ in }L^2.
\end{equation}
\end{lem}
\begin{proof}$ $\\
(i) This is a direct consequence of classical Strichartz estimates together with a fixed-point argument (see Lindblad and Sogge \cite{LSjfa95} or Keel and Tao \cite{KTajm98}).\\
(ii) Using \eqref{wconv} and compactness, we see that, up to extracting a subsequence still denoted by $z_n$, we have $(z_n(0), \partial_t z_n(0))\to (z(0), \partial_t z(0))$ as $n\to \infty$, strongly in $H^\gamma\times H^{\gamma-1}$ for any $\gamma<1$ close enough to $1$. Since we know from \cite{LSjfa95} that the Cauchy problem for equation \eqref{equ} is solved in $H^\gamma\times H^{\gamma-1}$, item (ii) follows.\\
(iii) This is a direct consequence of item (ii), thanks to classical Strichartz estimates and a fixed-point argument (see \cite{LSjfa95} or \cite{KTajm98}).\\
(iv) We will apply Lemma \ref{lemstri} with $v=z_n-z$ and $h=|z_n|^{p-1}z_n - |z|^{p-1}z$. Let us estimate in the following all the quantities appearing in the last line of Lemma \ref{lemstri}.\\
From \eqref{boundzzn}, we see that 
\begin{equation}\label{0}
\|v(0)\|_{\dot H^1}+\|\partial_t v(0)\|_{L^2}\le 4\bar M.
\end{equation}
It remains to estimate $\int_{S_\ttt} |D^{\frac 12}h|^{\frac {2(N+1)}{(N+3)}} dx dt$. 
Applying Lemma \ref{lemdd} with $U=z_n$, $V=z$ and $q=\frac{2(N+1)}{N+3}$, then integrating in time, we see that
\begin{align}
&\|D^{\frac 12} h\|_{L^{\frac{2(N+1)}{N+3}}(S_{\ttt})}^{\frac{2(N+1)}{N+3}}\nonumber\\
&\le \int_0^\ttt \|z_n\|_{L^{q_1}}^{{\frac{2(N+1)}{N+3}}(p-1)}\|D^{\frac 12} v\|_{L^{q_2}}^{\frac{2(N+1)}{N+3}}dt
+\int_0^\ttt \|v\|_{L^{q_1}}^{{\frac{2(N+1)}{N+3}}(p-1)}\|D^{\frac 12} z\|_{L^{q_2}}^{\frac{2(N+1)}{N+3}}dt\label{intt}
\end{align}
where
\begin{equation}\label{qr}
 \frac{N+3}{2(N+1)}=\frac {p-1}{q_1}+\frac 1{q_2}.
\end{equation}
Let us first handle the first term in \eqref{intt}.\\
{\it - The first term in \eqref{intt}}:\\ 
Using the condition \eqref{condp} on $p$, we see that
\[
 \frac{2(N+1)}{N+3}(p-1)<\frac{2(N+1)}{N+3}\frac 4{N-1} <\frac {2(N+1)}{N-2\gamma}.
\]
Since $\frac{N+3}{(p-1)(N-2\gamma)}<1$ from \eqref{condp}, using H\"older's inequality, we see that 
\begin{align}
&\int_0^\ttt \|z_n\|_{L^{q_1}}^{{\frac{2(N+1)}{N+3}}(p-1)}\|D^{\frac 12} v\|_{L^{q_2}}^{\frac{2(N+1)}{N+3}}dt\nonumber\\
\le& \left(\int_0^\ttt \|z_n\|_{L^{q_1}}^{\frac{2(N+1)}{N-2\gamma}} dt \right)^{\frac{(N-2\gamma)(p-1)}{N+3}}
\times \left(\int_0^\ttt \|D^{\frac 12} v\|_{L^{q_2}}^{\frac{2(N+1)}{N+3 - (N-2\gamma)(p-1)}}dt\right)^{\frac{N+3-(N-2\gamma)(p-1)}{N+3}}.\label{holder2}
\end{align}
Now, let us fix
\begin{equation}\label{defq1}
q_1 = \frac{2(N+1)}{N-2\gamma}.
\end{equation}
(note that $1<q_1<\infty$). From item (i) in Lemma \ref{propg}, we see that the first integral in \eqref{holder2} is bounded. 
As for the second integral in \eqref{holder2}, we note first that \eqref{qr} and \eqref{holder2} fix the value of $q_2$ such that
\begin{equation}\label{defq2}
\frac 1{q_2}=\frac{N+3}{2(N+1)}-\frac{(N-2\gamma)(p-1)}{2(N+1)},
\mbox{ hence }q_2 = \frac{2(N+1)}{N+3-(p-1)(N-2\gamma)}.
\end{equation}
therefore, that second integral is simply
\[
\|D^{\frac 12} v\|_{L^{q_2}(S_\ttt)}^{\frac{2(N+1)}{N+3}}
\]
I need to interpolate $\|D^{\frac 12} v\|_{L^{q_2}}$ between 
$\|v\|_{W^{\gamma-\frac 12,\frac{2(N+1)}{N-1}}}$
and $\|v\|_{H^1}$.
Applying Lemma \ref{lemgn} with $f=D^{\gamma-\frac 12} v$, $s_1=0$. $s_2=\frac 32-\gamma$, $s=1-\gamma$, $p=q_2$ defined in \eqref{defq2} and $p_1=\frac{2(N+1)}{N-1}$, we see from \eqref{deftheta} that
\[
\theta = \frac 1{3-2\gamma}\in(0,1)\mbox{ and }p_2 \sim \frac{4(1-\gamma)(N+1)}{4-(p-1)(N-2)}\mbox{ as }\gamma \to 1^-.
\]
Taking $\gamma$ close enough to $1$, we make $p_2\le 2$, and recalling that $\supp v \subset B(0,\bar A)$, we write from Lemma \ref{lemgn} and \eqref{boundzzn}: for any $t\in [0,\ttt]$, 
\begin{align}
\|D^{\frac 12} v(t)\|_{L^{q_2}}&\le C(\bar A)\|v(t)\|_{W^{\gamma-\frac 12,\frac{2(N+1)}{N-1}}}^\theta\|v(t)\|_{H^1}^{1-\theta}\nonumber\\
&\le C(\bar A, \bar M)\|v(t)\|_{W^{\gamma-\frac 12,\frac{2(N+1)}{N-1}}}^\theta\label{agn}
\end{align}
Since we have by definition \eqref{defq2} of $q_2$ that $\theta q_2 \le q_2 \le\frac{2(N+1)}{N-1}$ for $\gamma$ close enough to $1$, integrating \eqref{agn} in time, we see that
\[
\|D^{\frac 12} v\|_{L^{q_2}(S_\ttt)}^{q_2}
\le C(\bar A, \ttt, \bar M)\int_0^\ttt \|v(t)\|_{W^{\gamma-\frac 12,\frac{2(N+1)}{N-1}}}^{\frac{2(N+1)}{N-1}}dt\to 0.
\]
Using this together with items (i) and (iii) of Lemma \ref{propg}, we see from \eqref{holder2} that 
\begin{equation}\label{ligne1}
\int_0^\ttt \|z_n\|_{L^{q_1}}^{{\frac{2(N+1)}{N+3}}(p-1)}\|D^{\frac 12} v\|_{L^{q_2}}^{\frac{2(N+1)}{N+3}}dt\to 0\mbox{ as }n\to \infty.
\end{equation}
Now, I handle the second term in the right-hand side of \eqref{intt}.\\
{\it - The second term in \eqref{intt}}:\\ 
Using H\"older's inequality, we write
\begin{align*}
\int_0^\ttt \|v\|_{L^{q_1}}^{{\frac{2(N+1)}{N+3}}(p-1)}&\|D^{\frac 12} z\|_{L^{q_2}}^{\frac{2(N+1)}{N+3}}dt\\
\le&
\left(\int_0^\ttt \|D^{\frac 12} z\|_{L^{q_2}}^{\frac{2(N+1)}{N-1}}dt\right)^{\frac{N-1}{N+3}}
\left(\int_0^\ttt \|v\|_{L^{q_1}}^{{\frac{(p-1)(N+1)}2}}\right)^{\frac 4{N+3}}.
\end{align*}
Note from \eqref{defq1} and \eqref{defq2} that $q_1=\frac{2(N+1)}{N-2\gamma}$, $\frac{(p-1)(N+1)}2\le \frac{2(N+1)}{N-2\gamma}$ and $q_2 \le \frac{2(N+1)}{N-1}$ for $\gamma$ close enough to $1$. Since $\supp(v)$ and $\supp(z)$ are in $[0,\ttt]\times B(0, \bar A)$, it follows from items (i) and (iii) that 
\begin{equation}\label{ligne1bis}
\int_0^\ttt \|v\|_{L^{q_1}}^{{\frac{2(N+1)}{N+3}}(p-1)}\|D^{\frac 12} z\|_{L^{q_2}}^{\frac{2(N+1)}{N+3}}dt \to 0\mbox{ as }n\to \infty. 
\end{equation}
Using \eqref{ligne1}, \eqref{ligne1bis} 
and \eqref{intt}, we see that 
\[
\|D^{\frac 12} h\|_{L^{\frac{2(N+1)}{N+3}}(S_{\ttt})}^{\frac{2(N+1)}{N+3}}\to 0\mbox{ as }n\to \infty.
\]
Using \eqref{0}, Lemma \ref{lemstri} and item (ii) of Lemma \ref{propg}, we see that item (iv) of Lemma \ref{propg} follows.\\ 
(v) From \eqref{boundzzn} and the weak compactness of the unit ball in $H^1\times L^2$ together with the compactness of the embedding $H^1\times L^2 \subset H^\gamma\times H^{\gamma-1}$ for any $\gamma<1$, we see that for all $t\in [0,\ttt]$, 
\begin{align}
&(z_n(t), \partial_t z_n(t))\mbox{ converges to some }(\bar z(t), \bar z_1(t))\mbox{ as }n\to \infty,\label{one-hand}\\
&\mbox{weakly in }H^1\times L^2\mbox{ and strongly in }H^\gamma\times H^{\gamma-1}.\nonumber
\end{align}
Using item (ii) and the uniqueness of the limit in $ H^\gamma\times H^{\gamma-1}$, we see that
\[
(\bar z(t), \bar z_1(t))= (z(t), \partial_t z(t))
\]
 and item (v) follows from \eqref{one-hand}. This concludes the proof of Lemma \ref{propg}.
\end{proof}

\bigskip

{\bf Step 3: Weak continuity in $H^1\times L^2$ in some cone}

This step is devoted to the proof of Lemma \ref{lemcont2}. 

\begin{proof}[Proof of Lemma \ref{lemcont2}]
Lemma \ref{lemcont2} is simply a localized version of Lemma \ref{propg}.
Note first from \eqref{boundun} and \eqref{limun} that
\begin{equation}\label{boundz*}
\|(z^*, z_1^*)\|_{H^1\times L^2 (|x|<A)}\le 2 M_0.
\end{equation}
Therefore, we can define $u(\xixi,\tautau)$ as the maximal solution of \eqref{equ} with initial data $(z^*,z^*_1)$ defined in $\C_{\tautau^*}$ where $\tautau^*\le 1$ is maximal. Note from the solution of the Cauchy problem that 
\begin{equation}\label{alternative}
\mbox{either } \tautau^*=1\mbox{ or } \tautau^*<1 \mbox{ and
}\limsup_{\tautau \to \tautau^*}\|(u(\tautau), \partial_\tautau
u(\tautau))\|_{H^1\times L^2 (|\xixi|<A(1-\tautau))}= \infty.
\end{equation}
Introducing
\[
\tautau_\epsilon = \min(\tautau_0, \tautau^*-\epsilon).
\]
for any $\epsilon>0$, we claim that it is enough to prove that Lemma \ref{lemcont2} holds with $\tautau_0$ replaced by $\tautau_\epsilon$. In order to show this reduction, let us assume that Lemma \ref{lemcont2} holds with $\tautau_\epsilon$ instead of $\tautau_0$, and prove that $\tautau_\epsilon=\tautau_0$ for $\epsilon$ small enough
(in other words that $\tautau^*>\tautau_0$), which yields the good statement for Lemma \ref{lemcont2}.\\
 Assume by contradiction that $\tautau^*\le \tautau_0$. Then, using (b) and (c) of this lemma and \eqref{boundun}, we see that for all $\epsilon>0$, $\tautau_\epsilon=\tautau^*-\epsilon$ and for all $\tautau \in [0, \tautau^*-\epsilon]$, 
\begin{eqnarray*}
&&\|(u(\tautau), \partial_\tautau u(\tautau))\|_{H^1\times L^2(|\xixi|<A(1-\tautau))}
\le \|(u_n(\tautau), \partial_\tautau u_n(\tautau))\|_{H^1\times L^2(|\xixi|<A(1-\tautau))}\\
&+&\|(v_n(\tautau), \partial_\tautau v_n(\tautau))\|_{H^1\times L^2(|\xixi|<A(1-\tautau))} \le C(M_0,\tautau^*).
\end{eqnarray*}
Letting $\epsilon \to 0$, we see that $\limsup_{\tautau \to \tautau^*}\|(u(\tautau), \partial_\tautau
u(\tautau))\|_{H^1\times L^2 (|\xixi|<A(1-\tautau))}< \infty$. Since we also have $\tautau^*\le \tautau_0<1$, a contradiction follows from \eqref{alternative}.\\
Thus, as announced above, it is enough to prove that  Lemma \ref{lemcont2} holds with $\tautau_0$ replaced by $\tautau_\epsilon$.

\bigskip

Recall from \eqref{boundun} and \eqref{boundz*} that
\begin{equation}\label{bound0}
\|(u_n(0), \partial_\tautau u_n(0))\|_{H^1\times L^2 (|\xixi|<A)}\le 2M_0\mbox{ and }\|(u(0), \partial_t u(0))\|_{H^1\times L^2 (|x|<A)}\le 2M_0.
\end{equation}
From a classical method, we can extend all these functions to the ball $B(0,2A)$ such that if $|x|<A$, then
\begin{equation}\label{ext1}
u_n(x,0)=U_{0,n}(x),\;\;\partial_t u_n(x,0)=U_{1,n}(x),\;\;
u(x,0)=U_0(x),\;\;\partial_t u(x,0)=U_1(x),
\end{equation}
and
\[
\|(U_{0,n},U_{1,n})\|_{H^1\times L^2 (|x|<2A)}\le \gamma M_0\mbox{ and }
\|(U_0,U_1)\|_{H^1\times L^2 (|x|<2A)}\le \gamma M_0
\]
for some $\gamma =\gamma (A)>0$, with
\[
(U_{0,n},U_{1,n})\wto (U_0,U_1)\mbox{ in }H^1\times L^2(|x|<2A),\mbox{ as }n\to \infty.
\]         
Introducing $g\in C^\infty_c$ such that
\begin{equation}\label{defeta}
\forall |x|<\frac{4A}3,\;\;g(x)=1\mbox{ and }
\forall |x|>\frac{5A}3,\;\;g(x)=0
\end{equation}
and for all $x\in \m R^N$, 
\begin{equation}\label{ext2}
(Z_{0,n}(x), Z_{1,n}(x))=g(x)(U_{0,n}(x),U_{1,n}(x))\mbox{ and }(Z_0(x),Z_1(x))=g(x)(U_0(x),U_1(x)),
\end{equation}
we see that 
\[
\|(Z_{0,n},Z_{1,n})\|_{H^1\times L^2}\le \gamma'M_0\mbox{ and }
\|(Z_0,Z_1)\|_{H^1\times L^2}\le \gamma'M_0
\]
for some $\gamma'=\gamma'(A)>0$. Furthermore, 
\begin{equation}\label{zconv}
(Z_{0,n},Z_{1,n})\wto (Z_0,Z_1)\mbox{ in }H^1\times L^2,\mbox{ as }n\to \infty.
\end{equation}
Introducing $Z_n$ and $Z$ the solutions of \eqref{equ} defined in the whole space $\m R^N$, respectively with initial data $(Z_{0,n},Z_{1,n})$ and $(Z_0,Z_1)$, we see from the local existence theory in $H^1\times L^2$ that $Z$ and $Z_n$ are in $C([0,\ttt], H^1\times L^2$ and that
\begin{equation}\label{bound}
\forall t\in [0, \bar \ttt],\;\;
\|(Z(t), \partial_t Z(t))\|_{H^1\times L^2}+ 
\|(Z_n(t), \partial_t Z_n(t))\|_{H^1\times L^2}\le \bar M,
\end{equation}
for some $\ttt=\ttt(\gamma'M_0)>0$ and $\bar M=\bar M(\gamma'M_0)>0$. 
Since $\supp(Z_{0,n})\cup \supp (Z_0) \subset B(0, 2A)$, using the finite speed of propagation, we see that
\begin{equation}\label{supp}
\forall t\in [0, \bar \ttt],\;\;
\supp(Z_n(t))\cup \supp(Z(t))\subset B(0,\bar A)\mbox{ where }\bar A= 2A+\ttt,
\end{equation}
 From \eqref{bound}, \eqref{supp} and \eqref{zconv}, we can apply Lemma \ref{propg} and see that
\begin{align}
&\forall t\in [0,\ttt],\;\;(Z_n(t),\partial_t Z_n(t))\wto (Z(t),\partial_t Z(t))\mbox{ in }H^1\times L^2,\label{convrn}\\
&\sup_{t\in [0,\ttt]}\|Z_n(t)-Z(t)\|_{L^2}\to 0\mbox{ as }n\to \infty,\nonumber\\
&\sup_{t\in [0,\ttt]}\|(Z_n(t), \partial_t Z_n(t))-(Z(t), \partial_t Z(t))\|_{\dot H^1\times L^2}\le 5 \bar M,\nonumber
\end{align}
for all $n\ge n_0(\bar A, \ttt,\bar M)$.\\
Since we have from \eqref{ext1}, \eqref{defeta} and \eqref{ext2} for all $|x|<A$, 
\[
(Z_{0,n}(x), Z_{1,n}(x))=(u_n(x,0),\partial_t u_n(x,0))\mbox{ and }(Z_0(x),Z_1(x))=(u(x,0),\partial_t u(x,0)),
\]
and for all $t\ge 0$, $A(1-t)\le A-t$, it follows from the finite speed of propagation that for all $t\in [0,t_1]$ and $|x|<A(1-t)$, 
\[
Z_n(x,t)=u_n(x,t)\mbox{ and }Z(x,t)=u(x,t)\mbox{ where }t_1= \min(\ttt, t_\epsilon).
\]
Using \eqref{convrn}, we see that for all $t\in [0,t_1]$, 
\begin{align}
&(u_n(t),\partial_t u_n(t))\wto (u(t),\partial_t u(t))\mbox{ in }H^1\times L^2(|x|<A(1-t))\label{conv0}\\
&\sup_{t\in [0,t_1]}\|u_n(t)-u(t)\|_{L^2(|x|<A(1-t))}\to 0\mbox{ as }n\to \infty,\nonumber\\
&\sup_{t\in [0,t_1]}\|(u_n(t), \partial_t u_n(t))-(u(t), \partial_t u(t))\|_{\dot H^1\times L^2}\le 5 \bar M,\nonumber
\end{align}
for all $n\ge n_0(\bar A, \ttt,\bar M)$.\\
If $t_1= t_\epsilon$, then we are done.\\
If $t_1<t_\epsilon$, 
then we need to iterate this process. Introducing for all $\tau \in [0,\frac{t_\epsilon-t_1}{1-t_1}]$, 
\begin{align}
&\tilde u_n(\xi,\tau)=(1-t_1)^{\frac 2{p-1}}u_n(\xi(1-t_1),\tau(1-t_1)+t_1),\label{deftilde}\\
&\tilde u(\xi,\tau)=(1-t_1)^{\frac 2{p-1}}u(\xi(1-t_1),\tau(1-t_1)+t_1),\nonumber
\end{align}
we see from \eqref{boundun} that 
\[
\|(\tilde u_n(0), \partial_\tau \tilde u_n(0))\|_{H^1\times L^2 (|\xi|<A)}\le 2M_0\mbox{ and }\|(\tilde u(0), \partial_\tau \tilde u(0))\|_{H^1\times L^2 (|\xi|<A)}\le 2M_0,
\]
with the same bound as in \eqref{bound0}. By the same construction, we see that $(\tilde u_n, \partial_\tau \tilde u_n)  \wto (\tilde u, \partial_\tau \tilde u)$ as in \eqref{conv0}, provided that $\tau \le \bar t$, the same $\ttt$ as in the first iteration. Translating this result for $u_n$, we see that \eqref{conv0} holds for all $t\in[0, t_2]$ where 
\begin{equation}\label{deftk}
\forall k \ge 1,\;\;t_{k+1} = \min(t_\epsilon, t_k+(1-t_k)\ttt).
\end{equation}
If $t_2= t_\epsilon$, then we are done, otherwise, we further iterate the process and extend the convergence in \eqref{conv0} up to $t_k$ for some $k\ge 1$, provided that $t_{k'}<t_\epsilon$ whenever $k'\le k-1$. Clearly, in order to conclude, it is enough to prove that
\begin{equation}\label{hadaf}
t_k =t_\epsilon\mbox{ for some }k\ge 1.
\end{equation}
Assume this is not the case. From \eqref{deftk}, we see that for all $k\ge 1$, we have $t_k< t_\epsilon$ and $t_{k+1}=t_k+(1-t_k)\ttt$, hence  $t_{k+1}\ge t_k+(1-t_\epsilon)\ttt$ and $t_k \ge t_1+(k-1)(1-t_\epsilon)\ttt\to \infty$ as $k \to \infty$. Contradiction. Hence, \eqref{hadaf} holds and Lemma \ref{lemcont2} is proved. 
\end{proof}
 
\section{Properties of the Lyapunov functional $E$ \eqref{defenergy}}
We recall in this section some properties of  the Lyapunov functional $E$ \eqref{defenergy}, needed for the proof of the continuity of the blow-up time of equation \eqref{equ} with respect to initial data stated in Lemma \ref{lemcont}. 
These are the properties of $E$ we will need:
\begin{lem}[Properties of the functional $E$]\label{lemcrit}$ $\\
(i) {\bf (Blow-up criterion for equation \eqref{eqw})} Consider $W(y,s)$ a solution to equation \eqref{eqw} such that $W(y,s)$ is defined for all $|y|<1$ and 
$E(W(s_0))<0$ for some $s_0\in \m R$.
Then, $W(y,s)$ cannot exist for all $(y,s)\in B(0,1)\times [s_0, \infty)$.\\
(ii) {\bf (Continuity identity for the functional $E$)} 
For any solutions $w_1$ and $w_2$ of equation \eqref{eqw} and times $s_1$ and $s_2$ in their domains, we have
\begin{align*}
&|E(w_1(s_1))-E(w_2(s_2))|\\
&\le C\left(1+\left\|\vc{w_1(s_1)}{\partial_s w_1(s_1)}\right\|_{\q H}^p+\left\|\vc{w_2(s_2)}{\partial_s w_2(s_2)}\right\|_{\q H}^p\right)\left\|\vc{w_1(s_1)-w_2(s_2)}{\partial_s w_1(s_1)-\partial_s w_2(s_2)}\right\|_{\q H}.
\end{align*}
(iii) {\bf (Behavior of
$E(w_-(s))\to -\infty$
)} It holds that
$E(w_-(s))\to - \infty$
as $s\to \log(1-|\hat d(0)|)$.
\end{lem}
\begin{proof}$ $\\
(i) See Theorem 2 page 1147 in Antonini and Merle \cite{AMimrn01}.\\
(ii) The one-dimensional proof is given in Claim B.1 page 662 in \cite{MZisol10}.
The higher-dimensional case follows with the same proof: see the justification in Lemma E.1 in \cite{MZtds}.\\
(iii) The one-dimensional case is given in Appendix B page 85 in \cite{MZcmp08}. In higher dimensions, $w_-(y,s)$ depends only on one coordinate at most (along $\hat d(0)$ when $\hat d(0)\neq 0$), and we reduce to the one-dimensional case.
\end{proof}

\def\cprime{$'$} \def\cprime{$'$}

\noindent{\bf Address}:\\
Universit\'e de Cergy Pontoise, D\'epartement de math\'ematiques, 
2 avenue Adolphe Chauvin, BP 222, 95302 Cergy Pontoise cedex, France.\\
\vspace{-7mm}
\begin{verbatim}
e-mail: merle@math.u-cergy.fr
\end{verbatim}
Universit\'e Paris 13, Institut Galil\'ee, 
LAGA, 99 avenue J.B. Cl\'ement, 93430 Villetaneuse, France.\\
\vspace{-7mm}
\begin{verbatim}
e-mail: Hatem.Zaag@univ-paris13.fr
\end{verbatim}
\end{document}